\documentclass[11pt]{article}
\usepackage[utf8]{inputenc}
\usepackage[colorlinks=true,urlcolor=blue,
citecolor=red,linkcolor=blue,linktocpage,pdfpagelabels,bookmarksnumbered,bookmarksopen]{hyperref}
\usepackage[numbers]{natbib}

\usepackage{amsmath}
\usepackage{amssymb}
\usepackage{latexsym}
\usepackage{enumitem}
\usepackage{mathrsfs}
\usepackage{comment}
\usepackage{color}
\usepackage{bbm}

\newtheorem{theorem}{Theorem}[section]
\newtheorem{proposition}[theorem]{Proposition}
\newtheorem{corollary}[theorem]{Corollary}
\newtheorem{remark}[theorem]{Remark}
\newtheorem{lemma}[theorem]{Lemma}
\newtheorem{definition}[theorem]{Definition}
\newtheorem{assumption}[theorem]{Assumption}

\def\qed{ \ \vrule width.2cm height.2cm depth0cm\smallskip}
\newenvironment{proof}{\noindent {\bf Proof.\/}}{$\qed$\vskip 0.1in}

\newcommand{\ba}{\begin{array}}
\newcommand{\ea}{\end{array}}
\newcommand{\be}{\begin{equation}}
\newcommand{\ee}{\end{equation}}
\newcommand{\bea}{\begin{eqnarray}}
\newcommand{\eea}{\end{eqnarray}}
\newcommand{\beaa}{\begin{eqnarray*}}
\newcommand{\eeaa}{\end{eqnarray*}}

\def\no{\noindent}

\def\bs{\bigskip}

\def\qed{ \hfill \vrule width.25cm height.25cm depth0cm\smallskip}

\newcommand{\basa}{\begin{assumption}}
\newcommand{\easa}{\end{assumption}}

\newcommand{\bas}{\begin{assum}}
\newcommand{\eas}{\end{assum}}

\def\1{{\bf 1}}

\def\:{\!:\!}

at 9pt

\numberwithin{equation}{section}

\begin{document}

\title{\bf  An It\^o-type formula for some measure-valued processes and its application on controlled superprocesses}
\author{Shang Li}

\author{
Shang Li\thanks{\noindent Department of
Mathematics, Nankai University; email: 1944755529@qq.com.}
}

\date{\today}
\maketitle

\begin{abstract}
We derive an Itô-type formula for a measure-valued process that has a decomposition analogous to a classical semimartingale. The derivation begins with a time partitioning approach similar to the classical proof of Itô's formula. To address the new challenges arising from the measure-valued setting, we employ symmetric polynomials to approximate the second-order linear derivative of the functional on finite measures, alongside certain localization techniques.

A controlled superprocess with a binary branching mechanism can be interpreted as a weak solution to a controlled stochastic partial differential equation (SPDE), which naturally leads to such a decomposition. Consequently, this Itô-type formula makes it possible to derive the Hamilton-Jacobi-Bellman (HJB) equation and the verification theorem for controlled superprocesses with a binary branching mechanism. Additionally, we propose a heuristic definition for the viscosity solution of an equation involving derivatives on finite measures. We prove that a continuous value function is a viscosity solution in this sense and demonstrate the uniqueness of the viscosity solution when the second-order derivative term on the measure vanishes.
 \end{abstract}

 \bs
\no{\bf Keywords.} It\^o's formula, measure-valued processes, controlled superprocesses, linear functional derivative

\section{Introduction}\label{sec1}

It\^o's formula is a cornerstone in stochastic calculus, characterizing the infinitesimal change of a function of a semimartingale process, and establishing a connection between partial differential equations (PDEs) and stochastic processes \cite{bib13}. By advancing the Hamilton-Jacobi-Bellman (HJB) equation, the verification theorem, and the existence and uniqueness of viscosity solutions, It\^o's formula has become a key tool in stochastic control theory \cite{bib10, bib11, bib12}.

With recent developments in mean-field games and mean-field control theory, an It\^o-type formula for flows of probability measures of semimartingales was established using the tools of Lions derivatives and linear derivatives \cite{bib16, bib9, bib8}. This It\^o-type formula is essential for the derivation of the master equation in mean-field games and for the dynamic programming approach to the McKean-Vlasov control problem.

In certain situations, it is insufficient to merely consider a flow of probability measures of a process. For instance, when modeling the evolution of large populations, it can be more appropriate to consider a measure-valued process, such as a particle system \cite{bib6}, a Dawson-Watanabe process \cite{bib7, bib18}, or a Fleming-Viot process \cite{bib5}. However, to the best of our knowledge, no comparable results exist for processes that take values in finite measures. The methods used in mean-field theory encounter two major challenges. First, functions on the space of finite measures cannot be extended to functions on a Hilbert space, which is necessary to define the Lions derivative. Second, due to the randomness of the values, it seems impossible to approximate a finite measure by empirical measures.

It has been shown in \cite{bib3, bib4} that a super-Brownian motion with a binary branching mechanism can be regarded as a weak solution to a stochastic partial differential equation (SPDE) driven by space-time white noise. Recently, \cite{bib14} made pioneering contributions by formalizing controlled superprocesses with a binary branching mechanism through the martingale problem approach. In this work, the generalized martingale problem and the HJB equations were derived using Lions derivatives and linear derivatives. \cite{bib14} also proved that a controlled superprocess, which does not necessarily have a density, can be regarded as a weak solution to an SPDE driven by a martingale measure. The SPDE resulting from this controlled martingale problem provides a semimartingale-like decomposition for measure-valued processes. This suggests that an It\^o-type formula for certain measure-valued processes can be established using standard methods involving linear derivatives.

Our work aims to establish an It\^o-type formula for measure-valued processes that have a form resembling the classical semimartingale decomposition. From this It\^o-type formula, we immediately derive the generalized martingale problem for controlled superprocesses with a binary branching mechanism \cite{bib14} in new way. We also provide a formalization of controlled superprocesses following the framework of \cite{bib1, bib2, bib15}, which differs from the formalization in \cite{bib14}. In this framework, we establish the dynamic programming principle using the measurable selection theorem, following the approach of \cite{bib1, bib2}. This provides a new method for deriving the HJB equation and the verification theorem for controlled superprocesses from the dynamic programming principle. Additionally, this type of It\^o-type formula allows us to prove that the value function is the viscosity solution to the HJB equation. We also hope that this It\^o-type formula will serve as a useful tool for studying other measure-valued processes and SPDE systems.

A new challenge arises when proving this It\^o-type formula: when using Taylor's expansion for linear derivatives, the sum of the second-order terms is no longer a discrete version of a quadratic variation. This obstacle is overcome by using a linear combination of products of unary functions to approximate the second-order linear derivatives. This is feasible when the second-order linear derivatives are symmetric and continuous. We then employ localization techniques to rigorously justify the approximation procedure.

The paper is organized as follows. Section \ref{sec2} presents some notations and preliminaries for the concepts and tools used in this article. In Section \ref{sec3}, we prove the It\^o-type formula for certain measure-valued processes and derive some corollaries. In Section \ref{sec4}, we discuss its application to the HJB equation and the verification theorem for controlled superprocesses. Section \ref{sec5} discusses the definition, existence and uniqueness of viscosity solutions for a measure-valued control problem.

\section{Notations and Preliminaries}\label{sec2}

In this subsection, we introduce the space of finite measures and provide the necessary topological and functional background required for the development of later sections.

\subsection{Space of finite measures}

Given a Polish space $(E,d)$, let $\mathscr{B}(E)$ denote the Borel $\sigma$-algebra, which is the $\sigma$-algebra generated by the open sets. Let $C_b(E)$ denote the space of bounded continuous functions on E. We use $M(E)$ to denote the set of finite measures on the space $(E, \mathscr{B}(E))$. For some $\mu\in M(E)$ and $f\in C_b(E)$, we write $\mu(f)$ or $\langle f,\mu\rangle$ for the integral $\int_E f(x)\mu(dx)$. We equip $M(E)$ the weak topology $\tau$, which is the weakest topology to make $F_f(\mu)=\langle \mu, f \rangle$ continuous for all $f\in C_b(E)$. For some $f\in C(E)$ and $r>0$. Let $U(f,r)= \{\mu\in M(E): -r<\mu(f)<r\}$. Then it can be verified that $\{U(f,r)| f\in C_b(E), r>0\}$ is a subbasis of $\tau$.

When $E$ is $\mathbb{R}^d$, we can find a countable sequence $(\varphi_k)_{k\leq 0}$  in $C^2_b(\mathbb{R}^d)$ with $\varphi_0=1$ such that $\|\varphi_k\|_\infty\leq 1$ and $\varphi_k$ is dense in $\{f\in C_b(\mathbb{R}^d) :  \|f\|_\infty \leq 1\}$ (see, e.g., \cite{bib18}). We can define a distance on $M(\mathbb{R}^d)$, 
$$d(\mu,\lambda)=(\sum_{k=0}^\infty \frac{1}{2^k q_k}\langle\varphi_k, \mu-\lambda\rangle^2)^{\frac{1}{2}},$$
where $q_k=\max\{1,\|D\varphi_k\|^2_\infty,\|D^2\varphi_k\|^2_\infty\}.$ It can be verified that the topology induced from distance $d$ is equal to $\tau$(see, e.g. \cite{bib18}). We use the 2-norm instead of the 1-norm for its better differential properties, which will be discussed later.

When $E=\mathbb{R}$, let $\rho^N$ be a smooth cut-off function taking values in $[0,1]$ such that 
\begin{equation}
    \rho^N(x)=\left\{
    \begin{aligned}
        &1 ,\quad x\in [-N,N];\\
        &0 ,\quad x\in [-(N+1),N+1]^c.
    \end{aligned}
    \right.
\end{equation}
Denote $\mu^N$ the measure $\rho^n(x)\mu(dx)$.

\subsection{Linear functional derivative}
Linear functional derivative and Lions derivative were two critical tools to deal with the differential property of the space of probability measures in the mean field theory. The first can also be a good tool for the space of finite measures as the discussions in \cite{bib14}. More details can be found in \cite{bib16}.

\begin{definition}\label{definition3}
    For a function $f:M(E)\to \mathbb{R}$, the linear derivative of $f$ is a function $\frac{\delta f}{\delta \mu}:M(E)\times E\to \mathbb{R}$ such that for any $\mu$, $\lambda\in M(E)$,
    $$f(\mu)-f(\lambda)=\int_0^1\int_E \frac{\delta f}{\delta \mu}(\lambda+t(\mu-\lambda),x)(\mu-\lambda)(dx)dt.$$
\end{definition}
If the linear derivative of $f$ is bounded and continuous for the product topology, we say that $f\in C_b^1(M(E))$. If we additionally have that $\frac{\delta f}{\delta \mu}$ is continuously differentiable in some sense with respect to the spatial components, we say that $f\in C^{1,1}(M(E))$. More generally, we denote $C^{m,n}(M(E))$ to be the set of functions that are $m$ times differentiable with respect to the measure, and is $n$ times differentiable with respect to the spatial components for any order (less than $m$) of its derivative with respect to its measure. If additionally we have all these derivatives with respect to the spatial components are bounded, we write $f\in C^{m,n}_b$.

\subsection{Martingale measures}
The concept of martingale measure was introduced by Walsh in \cite{bib19}. It is proved in \cite{bib20} that some type of superprocesses can be viewed as solutions of SPDEs driven by space-time white noise. A controlled superprocess, which may not has a density, can be viewed as the solution a SPDE driven by a martingale measure \cite{bib14}. 

\begin{definition}\label{definition1}
    Let $(E,\mathscr{E})$ be a Lusin space, i.e a measurable space measurably isomorphic to a Borel set of a compact metric space. Let $(\Omega,\mathscr{F},\mathbb{P})$ be a probability space. Then a function $U: \mathscr{E}\times\Omega\to\mathbb{R}$ is called a $L^2$-valued measure if
    \begin{itemize}
        \item $E[U(A)^2]<\infty$, for any $A\in\mathscr{E}$;
        \item $U(\cdot,\omega)$ is $\sigma$-additive in $L^2$-norm, i.e. for a sequence of disjoint sets $A_n$ in $\mathscr{E}$ such that 
        $$\cup_{i=1}^n A_n=A,$$
        we have
        $$\lim_{n\to\infty}\sum_{i=1}^n U(A_n,\cdot)=U(A,\cdot)$$
        in the sense of $L^2$-norm.
    \end{itemize}
\end{definition}

     We denote $\|U(A)\|_2=[E(U(A)^2)]^{\frac{1}{2}}$ to be the $L^2$-norm of $U(A)$. A $L^2$-valued measure is called $\sigma$-finite if there exists a sequence of $E_n\in\mathscr{E}$ whose union is $E$, such that for each $n$,
    $$\sup_{A\in\mathscr{E}\cap E_n}\|U(A)\|_2<\infty.$$

\begin{definition}
     Let $(\Omega,\mathscr{F},\mathbb{P},(\mathscr{F}_t)_{t\leq 0})$ be a filtered probability space. Then $\mathcal{M}:[0,T]\times\Omega\times\mathscr{E}$ is called a martingale measure is
    \begin{itemize}
        \item $\mathcal{M}_0(A)=0$, a.s. for any $A\in\mathscr{E}$;
        \item $\mathcal{M}_t(\cdot)$ is a $\sigma$-finite, $L^2$-valued measure for all $t\in [0,T]$;
        \item $\mathcal{M}_{\cdot}(A)$ is a $(\mathscr{F}_t)_{t\leq 0}$ martingale for any $A\in\mathscr{E}$.
    \end{itemize}
\end{definition}

A martingale measure is called orthogonal if for any two disjoint sets $A_1$, $A_2$, $\mathcal{M}_t(A_1)$ and $\mathcal{M}_t(A_2)$ are two orthogonal martingales. From now on, all martingale measures are supposed to be orthogonal and finite.

Let $Q:\mathscr{E}\times [0,T]\to \mathscr{R}$ be a function such that $Q(A,t)= \langle \mathcal{M}(A),\mathcal{M}(A)\rangle_t$. Then there exists a random measure $Q$ (with some abuse of notation) called the density measure of $\mathcal{M}$ on $(E\times[0,T], \mathscr{E}\times\mathscr{B}([0,T]))$ such that 
$$Q(A,(s,t])=Q(A,t)-Q(A,s)$$ 
for all $A\in\mathscr{E}$ and $0\leq s<t\leq T$ .

We then introduce the procedure in \cite{bib19} of the establishment of the stochastic integral for martingale measures, which is similar to the procedure for Ito integral.

\begin{definition}
    A function $f:E\times [0,T]\times\Omega\to\mathbb{R}$ is called elementary if it has the form:
    $$f(x, t, \omega)= X(\omega) I_{(a,b]}(t) I_{A}(x),$$
    where $0\leq a<b\leq T$, $A\in\mathscr{E}$, $X\in\mathscr{F}_a$ and is bounded. A function is called simple if it is a finite sum of some elementary functions. We denote the set of simple functions by $\mathscr{S}$.
\end{definition}

Let $\mathscr{G}$ be the $\sigma$-algebra on $E\times [0,T]\times\Omega$ generated by $\mathscr{S}$. Let $\mathscr{P}_{\mathcal{M}}$ be the set of $\mathscr{G}$-measurable functions such that 
$$E\int_{[0,T]}\int_E(f(x,t))^2 Q(dx dt)<\infty.$$

We first define the stochastic integral for an elementary function $$f(x, t, \omega)= X(\omega) I_{(a,b]}(s) I_{A}(x)$$ to be a martingale measure $f\cdot \mathcal{M}$ such that
$$(f\cdot \mathcal{M})_t(B)= X(\omega)(\mathcal{M}_{t\wedge b}(A\cap B))-\mathcal{M}_{t\wedge a}(A\cap B).$$
We can then define the stochastic integral for the functions in $\mathscr{S}$ by linearity.
Then we can define the stochastic integral for the functions in $\mathscr{P}_\mathcal{M}$
by the $L^2$-limits since $\mathscr{S}$ is dense in $\mathscr{P}_{\mathcal{M}}$.

\begin{proposition}
    For a function $f\in\mathscr{P}_{\mathcal{M}}$, $f\cdot \mathcal{M}$ is a martingale measure with density measure $f(x,t)^2 Q(dx dt).$
\end{proposition}

\subsection{Controlled superprocesses}

The distribution of a super-diffusion process with Lipschitz coefficients $b,\sigma$ and binary branching mechanism $\phi(x,\varphi)=\gamma(x)\varphi^2(x)$ is the solution of the following martingale problem.

\begin{equation}
    F_{\varphi}(\mu_s)-F_{\varphi}(\mu_t)-\int_t^s\int_{\mathbb{R}}\mathscr{L}F_\varphi (x,\mu_u)\mu_u(dx)du
\end{equation}
is a martingale in some probability space $(\Omega,(\mathscr{F}_t)_{t\geq 0},P)$
where $F_\varphi(\lambda):=F(\left<\varphi,\lambda\right>)$, for any $F,\varphi\in C^2_b(\mathbb{R})$ and $\lambda\in M(\mathbb{R})$. The generator $\mathscr{L}$ is defined by

$$\mathscr{L}F_\varphi(x,\lambda)=F'(\left<\varphi,\lambda\right>)L\varphi(x)+\frac{1}{2}F''(\left<\varphi,\lambda\right>)\gamma(x)\varphi^2(x),$$
where the generator $L$ is defined by

$$L\varphi(x)=b(x)\varphi'(x)+\frac{1}{2}\sigma^2(x)\varphi''(x).$$

It's natural to let the coefficients $b,\sigma,\gamma$ depend on the current global measure $\lambda$ and some control $a$, then a controlled superprocess can be formalized following the weak control approach in \cite{bib15}\cite{bib1}\cite{bib2} with a little modification since we would distribute a control to all infinitesimal particles in a measure as in\cite{bib14}.

Let $A$ be a compact subset of $\mathbb{R}$ that represents the set of actions. Define
$$(b,\sigma,\gamma):\mathbb{R}\times M(\mathbb{R})\times A \to \mathbb{R}\times\mathbb{R}\times\mathbb{R}^+$$
as bounded and continuous functions. Additionally, we suppose $b,\sigma$ and $\gamma$ are Lipschitz uniformly in a. Define the generator $L$ depending on the global measure $\lambda$ with control $a$ by
$$L\varphi(x,\lambda,a)=b(x,\lambda,a)\varphi'(x)+\frac{1}{2}\sigma^2(x,\lambda,a)\varphi''(x)$$
for any $\varphi\in C^2_b(\mathbb{R})$.
Also define the generator $\mathscr{L}$ with control $a$ by
$$\mathscr{L}F_\varphi(x,\lambda,a)=F'(\left<\varphi,\lambda\right>)L\varphi(x,
\lambda,a)+\frac{1}{2}F''(\left<\varphi,\lambda\right>)\gamma(x,\lambda,a)\varphi^2(x)$$
for any $F,\varphi\in C^2_b(\mathbb{R})$.

Consider the canonical space $D([0,T];M(\mathbb{R}))$. Let $\mathbb{M}$ be a subset of $M([0,T]\times \mathbb{R}\times A)$ whose projection on $[0,T]$ is Lebesgue measure.
Let $\Omega=D([0,T];M(\mathbb{R}))\times \mathbb{M}$, $\mathscr{F}_t$ be the $\sigma$-algebra generated by the canonical process and the measures $$\{m([0,s]\times B_1\times A_1): B_1\in\mathscr{R}, A_1\in\mathscr{B}(A),0\leq s\leq t\}.$$ Denote $\mu$ and $m$ as the projection from $\Omega$ to $D([0,T];M(\mathbb{R}))$ and $\mathbb{M}$.

Next we introduce the notation of the concatenation of probability measures on $\Omega$, which is useful in the proof dynamic principle. For $(t,w)\in [0,T]\times\Omega$, denote
    $$D^t_w := \{\omega\in\Omega: \mu_t)(\omega)=\mu_t(w)\}$$
and
    $$D_{t,w}:=\{\omega\in\Omega:(\mu_s, m_s(\phi))(\omega)=(\mu_s,m_s(\phi))(w), \forall s \in [0,t], \phi\in C_b([0,T]\times\mathbb{R}\times A)\},$$
where $m_s(\phi)=\int_0^s\int_{\mathbb{R}\times A}\phi(u,x,a)m(du,dx,da)$.
We denote the concatenation of two paths $w$ and $\omega$ such that $\omega\in D^t_w$ at time $t$ by $w\otimes_t \omega$, which is defined by
\begin{equation}
    (\mu_s,m_s(\phi))(w\otimes\omega)=\left\{
    \begin{aligned}
        &(\mu_s,m_s(\phi))(w), s\in[0,t] \\
        &(\mu_s,m_s(\phi)-m_t(\phi))(\omega) + (0,m_t(\phi)(w)), s\in (t,T].
    \end{aligned}
    \right.
\end{equation}
    Then for a probability measure $\mathbb{P}$ and a family of $\mathscr{F}_\tau$- probability kernel $(\mathbb{Q}_w)_{w\in\Omega}$ such that $\mathbb{Q}_w(D^{\tau(w)}_w)=1$, we define the concatenated probability measure $\mathbb{P}\otimes_{\tau}\mathbb{Q}_\cdot$ by
    $$\mathbb{P}\otimes_{\tau}\mathbb{Q}_\cdot(B):=\int_\Omega\mathbb{P}(dw)\int_\Omega 1_B(w\otimes_{\tau(w)}\omega)\mathbb{Q}_w (d\omega).$$

\begin{definition}\label{definition2}
    Given some $(t,\lambda)\in [0,T]\times D([0,T],M(\mathbb{R}))$, a weak control rule with initial condition $(t,\lambda)$ is a probability measure $\mathbb{P}\in \mathscr{P}(\Omega)$ such that
    \begin{itemize}
        \item $\mathbb{P}(\{\mu_s=\lambda_s, 0\leq s\leq t\})=1$;
        \item $\mathbb{P}(\{\omega=(\mu, m): \exists \alpha\in B([0,T]\times\mathbb{R}; A) , m(ds, dx, da)=ds\mu_s(dx)\delta_{\alpha(s,x)}(da)\})=1$;
        \item $F_{\varphi}(\mu_s)-\int_t^s \int_{\mathbb{R}\times A} \mathscr{L}F_\varphi(x,\mu_u,a)m(da, dx, du)$ is a $(\mathbb{P},\mathscr{F})$ martingale for any $F, \varphi\in C^2_b(\mathbb{R})$ and $0\leq t\leq s\leq T$.
    \end{itemize}
    We denote this set $\mathscr{P}_{t,\lambda}$. 
\end{definition}

\begin{remark}
    In\cite{bib14}, a weak control is required additionally that there exist some $\{\mathscr{B}(\mathbb{R})\otimes\mathscr{F}_t\}_{t\in [0,T]}$-predictable $\alpha: [0,T]\times \mathbb{R}\times D([0,T]; M(\mathbb{R}))\to A$ such that
    $$\mathbb{P}(\{(\mu, m):  m(ds, dx, da)=ds\mu_s(dx)\delta_{\alpha(s,x)}(da)\})=1.$$ This is different from the second requirement in Definition\ref{definition2}. \cite{bib14} claims that a weak control rule exists given any $\{\mathscr{B}(\mathbb{R})\otimes\mathscr{F}_t\}_{t\in [0,T]}$-predictable $\alpha$, which is invalid to the best of our knowledge because we don't know how to solve a SDE with Borel coefficients. The method in \cite{bib14} is valid when $\alpha$ takes a fixed value which guarantees the set of weak control rules for any initial condition is non-empty (This is different from the traditional martingale problem for superprocesses with binary mechanism since the coefficients of the diffusion process and the mechanism depend on the global measure). 
\end{remark}

 We can then consider a optimal control problem for superprocesses. Given two continuous functions $f:\mathbb{R}\times M(\mathbb{R})\times A\to \mathbb{R}$ and $g:M(\mathbb{R})\to \mathbb{R}$ such that there exists some $C, p, q>0$ satisfying
    $$f(x,\lambda,a)|\leq C(1+d(\lambda,0)^p),$$
    $$g(\lambda)\leq c(1+d(\lambda,0)^p),$$
\begin{definition}
The cost function $J:[0,T]\times D([0,T];M(\mathbb{R}))\times \mathscr{P}(\Omega)\to \mathbb{R}$ is defined as
    $$J(t,\lambda,\mathbb{P}):=E^{\mathbb{P}}[\int_t^T\int_{\mathbb{R}}f(x,\mu_u,a)m(du, dx, da)+g(\mu_T)].$$
The value function $V:[0,T]\times D([0,T];M(\mathbb{R}))\to\mathbb{R}$ is defined as
 $$V(t,\lambda):=\inf_{\mathbb{P}\in\mathscr{P}_{t,\lambda}}J(t,\lambda,\mathbb{P}).$$
\end{definition}

\begin{remark}
    The value function is well-defined since the expectation of the supremum of the $p$-norm of $\mu_u$ is finite, the proof of which can be seen in \cite{bib14}.
\end{remark}
 
\begin{remark}\label{remark1}
    We can formulate the control problem in another way. For some fixed initial condition $(t,\lambda)$, a controlled superprocess is a term 
    $$(\Omega,\mathbb{F},(\mathscr{F}_t)_{t\in[0,T]}, \mathbb{P}, \mu , \alpha)$$
    such that $\mu$ is a $(\mathscr{F}_t)_{t\in [0,T]}$-adapted process satisfying $\mu_{t\wedge\cdot}=\lambda_{t\wedge\cdot}$ and $\alpha: [0,T]\times \mathbb{R}\times \Omega\to A$ is .
    $$F_{\varphi}(\mu_s)-\int_s^t \int_{\mathbb{R}\times A} \mathscr{L}F_\varphi(x,\mu_u,a)m(da, dx, du)$$ 
    is a $(\mathbb{P},\mathbb{F})$ martingale for any $F, \varphi\in C^2_b(\mathbb{R})$ and $t\leq s\leq T$. It can be easily checked that the value functions of the two formulations equal. So we can regard the two formulation as the same.
\end{remark}

\begin{proposition}\label{proposition3}
    Given a weak control rule $(\Omega,\mathbb{F},\mathbb{P})$, there exists an natural extension $(\hat{\Omega},\hat{\mathbb{F}},\hat{\mathbb{P}})$ of $(\Omega,\mathbb{F},\mathbb{P})$ supporting a continuous $\hat{\mathbb{F}}$-martingale measure $\mathcal{M}$ with intensity measure $\mu_u(dx)du$ such that
    \begin{equation}\label{equation1}
    \begin{aligned}
    \left<f, \mu_t -\mu_s \right>=\int_s^t\int_\mathbb{R} Lf(x,\mu_u,\alpha_u(x))\mu_u(dx)du\\
    +\int_s^t \int_\mathbb{R} \sqrt{\gamma(x,\mu_u,\alpha_u(x))}f(x)\mathcal{M}(dx,du).
    \end{aligned}
    \end{equation}
for all $f\in C^\infty_b(\mathbb{R})$ and $s,t\in [0.T]$.
\end{proposition}
\begin{proof}
    The construction of this extension probability space can be seen in \cite{bib20}\cite{bib21}.
\end{proof}

\begin{remark}
    With Remark \ref{remark1} and Proposition \ref{proposition3}, we formulate a controlled superprocess as a weak solution of Equation \ref{equation1}. Particularly, when there is only a single point in the set of actions and the coefficients don't depend on the global measure, Equation \ref{equation1} degenerate to the famous result in \cite{bib4}. 
\end{remark}

\section{It\^o's formula for measure-valued processes}\label{sec3}

 In this section, we will establish an It\^o type formula for the measure-valued process that can be written as an SPDE driven by an martingale measure. Our work is inspired by the background and tools used in \cite{bib14} where it is established that a controlled superprocess with binary branching mechanism can be regarded as the weak solution of such a SPDE. So we will first derive the It\^o formula for that process.
 \subsection{It\^o's formula for controlled superprocesses}

\begin{theorem}[Ito's formula for controlled superprocesses]\label{theorem1}
    Suppose $(\mu_t)_{t\geq 0}$ is the (weak) solution of Equation \ref{equation1}, then for any $F\in C^{2,2}_b(M(\mathbb{R}))$ and $0\leq s<t\leq T$, we have
    \begin{equation}
    \begin{aligned}
          F(\mu_t)-F(\mu_s)=\int_s^t \int_\mathbb{R} L\frac{\delta}{\delta \mu} F(
        \mu_u,\cdot)(x,\mu_u,\alpha_u)\mu_u(dx)du\\
        +\frac{1}{2}\int_s^t\int_\mathbb{R} \gamma(x,\mu_u,\alpha_u(x)) \frac{\delta^2}{\delta \mu^2}F(\mu_u,x,x)\mu_u(dx)du\\
        +\int_s^t\int_\mathbb{R}\sqrt{\gamma(x,\mu_u,\alpha_u(x))}\frac{\delta}{\delta \mu}F(\mu_u,x)\mathcal{M}(dx,du)
    \end{aligned}
    \end{equation}
\end{theorem}

It's useful to notice some lemmas and propositions before we start to prove Theorem \ref{theorem1}.

\begin{proposition}
        Given $F\in C^{2,2}_b(M(\mathbb{R}))$,  then $\frac{\delta^2}{\delta\mu^2}F(\mu,\cdot,\cdot)$is symmetric for any $\mu\in M(\mathbb{R})$, i.e.,  for any $x,y\in\mathbb{R}$, $\frac{\delta^2}{\delta\mu^2}F(\mu,x,y)=\frac{\delta^2}{\delta\mu^2}F(\mu,y,x)$.
\end{proposition}
\begin{proof}
        Considering $f(s,t)=F(\mu+s\mu_1+t\mu_2)$ for arbitrary $\mu_1,\mu_2\in M(\mathbb{R})$. We have
        \begin{equation}
            \frac{\partial^2}{\partial s\partial t}f(0,0)=\int_{\mathbb{R}^2}\frac{\delta^2}{\delta \mu^2}F(\mu,x,y)\mu_2(x)\mu_1(dy)
        \end{equation}
        and 
        \begin{equation}
            \frac{\partial^2}{\partial t\partial s}f(0,0)=\int_{\mathbb{R}^2}\frac{\delta^2}{\delta \mu^2}F(\mu,x,y)\mu_1(x)\mu_2(dy).
        \end{equation}
        Since $f\in C^{2,2}$, we have 
        \begin{equation}
             \frac{\partial^2}{\partial s\partial t}f(0,0)=\frac{\partial^2}{\partial t\partial s}f(0,0).
        \end{equation}
        Let $\mu_1=\delta_x$, $\mu_2=\delta_y$, we obtain that $\frac{\delta^2}{\delta\mu^2}F(\mu,\cdot,\cdot)$is symmetric.
\end{proof}

\begin{lemma}[Stone-Weierstrass]\label{lemma1}
    Let $X$ be a compact Hausdorff space. Let $C(X)$ be the continuous real-valued functions on $X$ with the supremum norm. Let $\mathscr{A}$ be a subalgebra of $C(X)$ such that:
    \begin{itemize}
        \item for any $x\in X$, there is some $f\in \mathscr{A}$ such that $f(x)\neq 0$;
        \item if $x,y\in X$ and $x\neq y$, then there is an $f$ in $\mathscr{A}$ such that $f(x)\neq f(y)$;
    \end{itemize}
        then $\mathscr{A}$ is dense in $C(X)$.
\end{lemma}
The proof of Lemma \ref{lemma1} can be found in \cite{bib17}.

\begin{proposition}\label{proposition1}
    Let $K$ be a symmetric compact subset of $\mathbb{R}^2$. Suppose $f:[0,T]\times K\to\mathbb{R}$ is continuous, and $f(u,\cdot,\cdot)$ is symmetric for any $u\in [0,T]$. Then there exists a sequence of polynomials $f_n:[0,T]\times\mathbb{R}\times\mathbb{R}\to\mathbb{R}$ such that $f_n(u,\cdot,\cdot)$ is symmetric for any $u\in [0,T]$ and $f_n$ converges to $f$ uniformly in $[0,T]\times K$.
\end{proposition}
\begin{proof}
    Given any $f$ satisfying the conditions in Proposition \ref{proposition1}, let $\mathscr{A}$ be the set of polynomials. Applying Lemma \ref{lemma1}, we obtain
    that there exists a sequence of polynomials $g_n$ converges to $f$ uniformly in $[0,T]\times K$. Let $f_n(t,x,y)=\frac{1}{2} [g_n(t,x,y)+g_n(t,y,x)]$. Then $f_n$ is a sequence of symmetric polynomials and converges to $f$ uniformly in $[0,T]\times K$.
\end{proof}
\begin{proposition}\label{proposition2}
    Let $f:[0,T]\times\mathbb{R}\times\mathbb{R}\to\mathbb{R}$ be a polynomial such that $f(u,\cdot,\cdot)$ is symmetric for any $u\in [0,T]$. Then there exists some $m\in\mathbb{N}$ and polynomials $\lambda_j: [0,T]\to\mathbb{R}$ and $g_j:\mathbb{R}\to\mathbb{R}$, $j=1,2, \cdots,  m$, such that $$f(u,x,y)=\sum_{j=1}^m \lambda_j (u)g_j(x)g_j(y).$$
\end{proposition}
\begin{proof}
    It's obvious that this proposition holds true for the polynomials of degree 0. 
    Suppose Proposition \ref{proposition2} holds true for all the polynomials whose degree is less than n. For any $f$ of degree n, we have $x^n$ and $y^n$ have the same coefficient $a$. Since $x^n+y^n=(x^n+1)(y^n+1)-x^n y^n-1$, we have 
    $$f(u,x,y)=a[(x^n+1)(y^n+1)-x^n y^n-1]+xy[\tilde{f}(u,x,y)+\lambda(u)]$$, where $\tilde{f}$ is a polynomial of degree less than $n$ and $\lambda(u)$ is some polynomial of $u$. Applying the proposition for $\tilde{f}$, we have that the proposition holds true for $f$. So we finish our proof by induction.
\end{proof}

\subsection{proof of Theorem \ref{theorem1}}
\begin{proof}
We will first proof that for any $N\in\mathbb{N}$
\begin{equation}\label{equation3}
\begin{aligned}
    F(\mu^N_t)-F(\mu^N_s)=\int_s^t \int_\mathbb{R} L\frac{\delta}{\delta \mu} F(
        \mu^N_u,\cdot)(x,\mu_u,\alpha_u)\mu^N_u(dx)du\\
        +\frac{1}{2}\int_s^t\int_\mathbb{R} \gamma(x,\mu_u,\alpha_u) \frac{\delta^2}{\delta \mu^2}F(\mu^N_u,x,x)\mu^N_u(dx)du\\
        +\int_s^t\int_\mathbb{R}\sqrt{\gamma(x,\mu_u,\alpha_u(x))}\frac{\delta}{\delta \mu}F(\mu^N_u,x)\rho^N(x)\mathcal{M}(dx,du).
\end{aligned}
\end{equation}

Consider an increasing sequence $s=t^n_0\leq\cdots\leq t^n_{p_n}=t$ of subdivisions of $[s,t]$ whose mesh tends to $0$. We use $m^{N,n}_i$ (resp. $m^n_i$) to denote $\mu^N_{t^n_{i+1}}-\mu^N_{t^n_i}$ (resp. $\mu_{t^n_{i+1}}-\mu_{t^n_i}$). Then we have
\begin{equation}
    \begin{aligned}
        &F(\mu^N_t)-F(\mu^N_s)\\
        =& \sum_{i=1}^{p_n} [F(\mu^N_{t^n_{i+1}})-F(\mu^N_{t^n_{i}})]\\
        =&\sum_{i=1}^{p_n}[\int_\mathbb{R}\frac{\delta}{\delta\mu}F(\mu^N_{t^n_i},x)m^{N,n}_i(dx)\\
        +&\frac{1}{2}\int_{\mathbb{R}^2} \frac{\delta^2}{\delta\mu^2}F(\mu^N_{t^n_i}+\theta^n_i m^{N,n}_i,x,y)m^{N,n}_i(dx)m^{N,n}_i(dy)]\\
        =&\sum_{i=1}^{p_n}[\int_\mathbb{R}\frac{\delta}{\delta\mu}F(\mu^N_{t^n_i},x)\rho^N(x)m^n_i(dx)\\
        +&\frac{1}{2}\int_{\mathbb{R}^2} \frac{\delta^2}{\delta\mu^2}F(\mu^N_{t^n_i}+\theta^n_i m^{N,n}_i,x,y)\rho^N(x)\rho^N(y)m^n_i(dx)m^n_i(dy)]
    \end{aligned}
\end{equation}
, where $\theta^n_i\in [0,1]$. 

We are going to prove that
\begin{equation}\label{equation13}
    \begin{aligned}
        &\lim_{n\to\infty}\sum_{i=1}^{p_n}[\int_\mathbb{R}\frac{\delta}{\delta\mu}F(\mu^N_{t^n_i},x)\rho^N(x)m^n_i(dx)\\
        =&\int_s^t \int_\mathbb{R} L\frac{\delta}{\delta \mu} F(
        \mu^N_u,\cdot)(x,\mu_u,\alpha_u)\mu^N_u(dx)du\\
        +&\int_s^t\int_\mathbb{R}\sqrt{\gamma(x,\mu_u,\alpha_u(x))}\frac{\delta}{\delta \mu}F(\mu^N_u,x)\rho^N(x)\mathcal{M}(dx,du)
    \end{aligned}
\end{equation}
and that
\begin{equation}\label{equation14}
    \begin{aligned}
        &\lim_{n\to\infty}\sum_{i=1}^{p_n}\frac{1}{2}\int_{\mathbb{R}^2} \frac{\delta^2}{\delta\mu^2}F(\mu^N_{t^n_i}+\theta^n_i m^{N,n}_i,x,y)\rho^N(x)\rho^N(y)m^n_i(dx)m^n_i(dy)\\
        =&\frac{1}{2}\int_s^t\int_\mathbb{R} \gamma(x,\mu_u,\alpha_u) \frac{\delta^2}{\delta \mu^2}F(\mu^N_u,x,x)\mu^N_u(dx)du.
    \end{aligned}
\end{equation}
Both limits in Equation \ref{equation13} and \ref{equation14} are in the sense of probability.

Denote $\phi^n(u)=t^n_i, u\in [t^n_i,t^n_{i+1})$. Applying Equation \ref{equation1} to $\frac{\delta}{\delta\mu}F(\mu^N_{t^n_i},\cdot)\rho^N(\cdot)$, we have 
\begin{equation}\label{equation15}
\begin{aligned}
    &\sum_{i=1}^{p_n}\int_\mathbb{R}\frac{\delta}{\delta\mu}F(\mu^N_{t^n_i},x)\rho^N(x)m^n_i(dx)\\
    =&\sum_{i=1}^{p_n}[\int_{t^n_i}^{t^n_{i+1}}\int_\mathbb{R}L\frac{\delta}{\delta \mu} F(\mu^N_{t_i},\cdot)(x,\mu_u,\alpha_u)\mu^N_u(dx)du\\
    +&\int_{t^n_i}^{t^n_{i+1}}\int_\mathbb{R}\sqrt{\gamma(x,\mu_u,\alpha_u(x))}\frac{\delta}{\delta \mu}F(\mu^N_{t^n_i},x)\rho^N(x)\mathcal{M}(dx,du)]\\
    =&\int_s^t\int_\mathbb{R}L\frac{\delta}{\delta\mu}F(\mu^N_{\phi^n(u)},\cdot)(x,\mu_u,\alpha_u)\mu^N_u(dx)du\\
    +&\int_s^t\int_\mathbb{R}\sqrt{\gamma(x,\mu_u,\alpha_u(x))}\frac{\delta}{\delta \mu}F(\mu^N_{\phi^n(u)},x)\rho^N(x)\mathcal{M}(dx,du).
\end{aligned}
\end{equation}
Since $F\in C^{2,2}_b(M(\mathbb{R}))$ and all the coefficients are bounded, we can use dominated convergence theorem and get
\begin{equation}\label{equation16}
    \begin{aligned}
        \lim_{n\to\infty}\int_s^t\int_\mathbb{R}L\frac{\delta}{\delta\mu}F(\mu^N_{\phi^n(u)},\cdot)(x,\mu_u,\alpha_u)\mu^N_u(dx)du\\
        =\int_s^t\int_\mathbb{R}L\frac{\delta}{\delta\mu}F(\mu^N_{u},\cdot)(x,\mu_u,\alpha_u)\mu^N_u(dx)du\\
    \end{aligned}
\end{equation}
for almost all $\omega$. Notice that 
$$E[\int_s^t\int_\mathbb{R} 1 \mu_u(dx)du]=(t-s)\mu_0(\mathbb{R})<\infty.$$
Then use the bounds of the coefficients and the linear functional derivative, we can apply the property of martingale measures and the dominated convergence theorem then get 
\begin{equation}
     \begin{aligned}
        & \lim_{n\to\infty}E[\int_s^t\int_\mathbb{R}\sqrt{\gamma(x,\mu_u,\alpha_u(x))}\frac{\delta}{\delta \mu}F(\mu^N_{\phi^n(u)},x)\rho^N(x)\mathcal{M}(dx,du)\\
        -& \int_s^t\int_\mathbb{R}\sqrt{\gamma(x,\mu_u,\alpha_u(x))}\frac{\delta}{\delta \mu}F(\mu^N_u,x)\rho^N(x)\mathcal{M}(dx,du)]^2\\
       =& \lim_{n\to\infty}E\{\int_s^t\int_{\mathbb{R}}[\sqrt{\gamma(x,\mu_u,\alpha_u(x))}\frac{\delta}{\delta \mu}F(\mu^N_{\phi^n(u)},x)\rho^N(x)\\
       -& \sqrt{\gamma(x,\mu_u,\alpha_u(x))}\frac{\delta}{\delta \mu}F(\mu^N_u,x)]^2\mu^N_u(dx)du\}\\
       =& 0.
    \end{aligned}
\end{equation}
So we have
\begin{equation}\label{equation17}
\begin{aligned}
    \lim_{n\to\infty}\int_s^t\int_\mathbb{R}\sqrt{\gamma(x,\mu_u,\alpha_u(x))}\frac{\delta}{\delta \mu}F(\mu^N_{\phi^n(u)},x)\rho^N(x)\mathcal{M}(dx,du)\\
    =\int_s^t\int_\mathbb{R}\sqrt{\gamma(x,\mu_u,\alpha_u(x))}\frac{\delta}{\delta \mu}F(\mu^N_u,x)\rho^N(x)\mathcal{M}(dx,du)
\end{aligned}
\end{equation}
in probability.
Combining Equation \ref{equation15}, \ref{equation16} and \ref{equation17}, we obtain Equation \ref{equation13}. As a byproduct, we have the limit in the left hand side of Equation \ref{equation14} exists in the sense of probability, which is useful in the following proof.

Notice that
\begin{equation}\label{equation8}
    \begin{aligned}
        &\lim_{n\to\infty}\sum_{i=1}^{p_n}\int_{\mathbb{R}^2} 1 m^n_i(dx)m^n_i(dy)\\
        =&\lim_{n\to\infty}\sum_{i=1}^{p_n}[\int_{\mathbb{R}}1 m^n_i(dx)]^2\\
        =&\lim_{n\to\infty}\sum_{i=1}^{p_n}[\int_{t^n_i}^{t^n_{i+1}}\int_{\mathbb{R}}\sqrt{\gamma(x,\mu_u,\alpha_u(x))}\mathcal{M}(dx,du)]^2\\
        =&\int_s^t\int_{\mathbb{R}}\gamma(x,\mu_u,\alpha_u(x))\mu_u(dx)du
    \end{aligned}
\end{equation}
in probability.
Using the continuity of $\frac{\delta^2}{\delta\mu^2}F,\mu^N$, we have
\begin{equation}\label{equation9}
    \begin{aligned}
        \lim_{n\to\infty} \max_{1\leq i\leq p_n} \left| \frac{\delta^2}{\delta\mu^2}F(\mu^N_{t^n_i}+\theta^n_i m^{N,n}_i,x,y)-\frac{\delta^2}{\delta\mu^2}F(\mu^N_{t^n_i},x,y)\right|=0
    \end{aligned}
\end{equation}
for almost all $\omega$.

Given any $\varepsilon, \delta>0$, we can find some $1<C<\infty$ and $\Omega_1$ satisfying $P(\Omega_1)<\delta $ such that for each $\omega\in\Omega-\Omega_1$, 
\begin{equation}\label{equation10}
    \int_s^t\int_{\mathbb{R}}\gamma(x,\mu_u,\alpha_u(x))\mu_u(dx)du< C.
\end{equation}
And we can find some $n_1$ and $\Omega_2$ satisfying $P(\Omega_2)<\delta$ such that for any $n>n_1$ and $\omega\in\Omega-\Omega_2$ ,
\begin{equation}\label{equation11}
    \left| \sum_{i=1}^{p_n}\int_{\mathbb{R}^2} 1 m^n_i(dx) m^n_i (dy)-\int_s^t\int_{\mathbb{R}}\gamma(x,\mu_u,\alpha_u(x))\mu_u(dx)du\right|< \varepsilon
\end{equation}
Then we can find some $n_2$ and $\Omega_3$ satisfying $P(\Omega_3)<\delta$ such that
for any $n>n_2$ and $\omega\in\Omega-\Omega_3$,
\begin{equation}\label{equation12}
    \max_{1\leq i\leq p_n} \left| \frac{\delta^2}{\delta\mu^2}F(\mu^N_{t^n_i}+\theta^n_i m^{N,n}_i,x,y)-\frac{\delta^2}{\delta\mu^2}F(\mu^N_{t^n_i},x,y)\right|< \frac{\varepsilon}{C}.
\end{equation}
Note that although $m^n_i(dx)$ is a signal measure on $\mathbb{R}$, $m^n_i(dx)m^n_i(dy)$ is actually a measure on $\mathbb{R}^2$.
Combining Equation \ref{equation9}, \ref{equation11} and \ref{equation12}, we have
\begin{equation}
\begin{aligned}
     \lvert \sum_{i=1}^{p_n}\int_{\mathbb{R}^2} \frac{\delta^2}{\delta\mu^2}F(\mu^N_{t^n_i}+\theta^n_i m^{N,n}_i,x,y)\rho^N(x)\rho^N(y)  m^n_i(dx)m^n_i(dy)\\
     -\sum_{i=1}^{p_n}\int_{\mathbb{R}^2}\frac{\delta^2}{\delta\mu^2}F(\mu^N_{t^n_i},x,y)\rho^N(x)\rho^N(y)  m^n_i(dx)m^n_i(dy)\rvert < 2 \varepsilon
\end{aligned}
\end{equation}
for any $n>n_1\vee n_2$ and $\omega\in\Omega-(\Omega_1\cup\Omega_2\cup\Omega3)$.
In other word, we have 
\begin{equation}
    \begin{aligned}
        &\lim_{n\to\infty}\sum_{i=1}^{p_n}\int_{\mathbb{R}^2} \frac{\delta^2}{\delta\mu^2}F(\mu^N_{t^n_i}+\theta^n_i m^{N,n}_i,x,y)\rho^N(x)\rho^N(y)  m^n_i(dx)m^n_i(dy)\\
        =&\lim_{n\to\infty}\sum_{i=1}^{p_n}\int_{\mathbb{R}^2}\frac{\delta^2}{\delta\mu^2}F(\mu^N_{t^n_i},x,y)\rho^N(x)\rho^N(y)  m^n_i(dx)m^n_i(dy)
    \end{aligned}
\end{equation}
in the sense of probability.

Now fix some $\omega$, applying Proposition \ref{proposition1} and \ref{proposition2} to $$f(u,x,y)=\frac{\delta^2}{\delta\mu^2}F(\mu^N_{u},x,y)$$ 
on the compact set $[0,T]\times \mathcal{K}^N\times\mathcal{K}^N$, we obtain that there exists a sequence 
$$f_m(u,x,y)=\sum_{j=1}^{q_m}\lambda^m_j(u)g^m_j(x)g^m_j(y)$$
such that
\begin{equation}
    \lim_{m\to\infty} [f(u,x,y)-f_m(u,x,y)]=0
\end{equation}
uniformly in $u,x,y$. And we have
\begin{equation}
    \begin{aligned}
        &\lim_{n\to\infty}\sum_{i=1}^{p_n}\int_{\mathbb{R}^2}f_m(t^n_i,x,y)\rho^N(x)\rho^N(y)m^n_i(dx)m^n_i(dy)\\
        =&\lim_{n\to\infty}\sum_{i=1}^{p_n}\int_{\mathbb{R}^2}\sum_{j=1}^{q_m}\lambda^m_j(t^n_i)g^m_j(x)g^m_j(y)\rho^N(x)\rho^N(y)m^n_i(dx)m^n_i(dy)\\
        =&\lim_{n\to\infty}\sum_{i=1}^{p_n}\sum_{j=1}^{q_m}\lambda^m_j(t^n_i)\int_{\mathbb{R}^2}g^m_j(x)g^m_j(y)\rho^N(x)\rho^N(y)m^n_i(dx)m^n_i(dy)\\
        =&\lim_{n\to\infty}\sum_{i=1}^{p_n}\sum_{j=1}^{q_m}\lambda^m_j(t^n_i)[\int_{\mathbb{R}}g^m_j(x)\rho^N(x)m^n_i(dx)]^2\\
        =&\sum_{j=1}^{q_m}\lim_{n\to\infty}\sum_{i=1}^{p_n}\lambda^m_j(t^n_i)[\int_{t^n_i}^{t^n_{i+1}}\int_{\mathbb{R}}\sqrt{\gamma(x,\mu_u,\alpha_u(x))}g^m_j(x)\rho^N(x)\mathcal{M}(dx,du)]^2\\
        =&\sum_{j=1}^{q_m}\int_s^t\int_{\mathbb{R}}\gamma(x,\mu_u,\alpha_u(x))\lambda^m_j(u)[g^m_j(x)]^2\mu^N_u(dx)du\\
        =&\int_s^t\int_{\mathbb{R}}\gamma(x,\mu_u,\alpha_u(x))f_m(u,x,x)\mu^N_u(dx)du
    \end{aligned}
\end{equation}
in the sense of probability.

Given any $\varepsilon,\delta>0$,  we can find some $1<C<\infty$ and $\Omega_1$ satisfying $P(\Omega_1)<\delta $ such that for each $\omega\in\Omega-\Omega_1$, Equation \ref{equation10} is satisfied.
Then we can find some $m_2$ and $\Omega_2$ satisfying $P(\Omega_2)<\delta$ such that for any $\omega\in\Omega-\Omega_2$
\begin{equation}
    \left| f(u,x,y)-f_m(u,x,y) \right|<\frac{\varepsilon}{C}.
\end{equation}
Therefore we have
\begin{equation}
    \begin{aligned}
        &\lvert\lim_{n\to\infty}\sum_{i=1}^{p_n}\int_{\mathbb{R}^2}f(t^n_i,x,y)\rho^N(x)\rho^N(y)  m^n_i(dx)m^n_i(dy)\\
        -& \lim_{n\to\infty}\sum_{i=1}^{p_n}\int_{\mathbb{R}^2}f_m(t^n_i,x,y)\rho^N(x)\rho^N(y)m^n_i(dx)m^n_i(dy)\rvert\\
        < & \frac{\varepsilon}{C}\lim_{n\to\infty}\sum_{i=1}^{p_n}\int_{\mathbb{R}^2}1 m^n_i(dx)m^N_i(dy)\\
        < &\varepsilon
    \end{aligned}
\end{equation}
for any $m>m_1\vee m_2$ and $ \omega\in\Omega-(\Omega_1\cup\Omega_2)$, which gives that
\begin{equation}
    \begin{aligned}
        &\lim_{n\to\infty}\sum_{i=1}^{p_n}\int_{\mathbb{R}^2}f(t^n_i,x,y)\rho^N(x)\rho^N(y)  m^n_i(dx)m^n_i(dy)\\
        =&\lim_{m\to\infty}\lim_{n\to\infty}\sum_{i=1}^{p_n}\int_{\mathbb{R}^2}f_m(t^n_i,x,y)\rho^N(x)\rho^N(y)m^n_i(dx)m^n_i(dy)\\
        =&\lim_{m\to\infty}\int_s^t\int_{\mathbb{R}}\gamma(x,\mu_u,\alpha_u(x))f_m(u,x,x)\mu^N_u(dx)du\\
        =&\int_s^t\int_\mathbb{R} \gamma(x,\mu_u,\alpha_u) \frac{\delta^2}{\delta \mu^2}F(\mu^N_u,x,x)\mu^N_u(dx)du.
    \end{aligned}
\end{equation}

As \ref{equation3} is proved, we can get Theorem \ref{theorem1} by letting $N\to\infty$. Since $F$ is continuous, we have 
\begin{equation}\label{equation4}
    \lim_{N\to\infty}F(\mu^N_t)-F(\mu^N_s)=F(\mu_t)-F(\mu_s)
\end{equation}
for almost all $\omega$. Since $F\in C^{2,2}_b(M(\mathbb{R})) $ and $b, \sigma, \gamma$ are bounded, using dominated convergence theorem, we have
\begin{equation}\label{equation5}
    \begin{aligned}
        \lim_{N\to\infty}\{\int_s^t \int_\mathbb{R} L\frac{\delta}{\delta \mu} F(
        \mu^N_u,\cdot)(x,\mu_u,\alpha_u)\mu^N_u(dx)du\\
        +\frac{1}{2}\int_s^t\int_\mathbb{R} \gamma(x,\mu_u,\alpha_u) \frac{\delta^2}{\delta \mu^2}F(\mu^N_u,x,x)\mu^N_u(dx)du\}\\
        =\int_s^t \int_\mathbb{R} L\frac{\delta}{\delta \mu} F(
        \mu_u,\cdot)(x,\mu_u,\alpha_u)\mu_u(dx)du\\
        +\frac{1}{2}\int_s^t\int_\mathbb{R} \gamma(x,\mu_u,\alpha_u(x)) \frac{\delta^2}{\delta \mu^2}F(\mu_u,x,x)\mu_u(dx)du
    \end{aligned}
\end{equation}
for almost all $\omega$. As for the third part, we have
\begin{equation}\label{equation6}
    \begin{aligned}
        & \lim_{N\to\infty}E[\int_s^t\int_\mathbb{R}\sqrt{\gamma(x,\mu_u,\alpha_u(x))}\frac{\delta}{\delta \mu}F(\mu^N_u,x)\rho^N(x)\mathcal{M}(dx,du)\\
        -& \int_s^t\int_\mathbb{R}\sqrt{\gamma(x,\mu_u,\alpha_u(x))}\frac{\delta}{\delta \mu}F(\mu_u,x)\mathcal{M}(dx,du)]^2\\
       =& \lim_{N\to\infty}E\{\int_s^t\int_{\mathbb{R}}[\sqrt{\gamma(x,\mu_u,\alpha_u(x))}\frac{\delta}{\delta \mu}F(\mu^N_u,x)\rho^N(x)\\
       -& \sqrt{\gamma(x,\mu_u,\alpha_u(x))}\frac{\delta}{\delta \mu}F(\mu_u,x)]^2\mu_u(dx)du\}\\
       =& 0.
    \end{aligned}
\end{equation} 
So we have
\begin{equation}\label{equation7}
\begin{aligned}
    \lim_{N\to\infty}\int_s^t\int_\mathbb{R}\sqrt{\gamma(x,\mu_u,\alpha_u(x))}\frac{\delta}{\delta \mu}F(\mu^N_u,x)\rho^N(x)\mathcal{M}(dx,du)\\
    =\int_s^t\int_\mathbb{R}\sqrt{\gamma(x,\mu_u,\alpha_u(x))}\frac{\delta}{\delta \mu}F(\mu_u,x)\mathcal{M}(dx,du)
\end{aligned}
\end{equation}
in probability. Adding \ref{equation3}, \ref{equation4}, \ref{equation5} and \ref{equation7}, we have
\begin{equation}
    \begin{aligned}
        &F(\mu_t)-F(\mu_s)\\
        =&\lim_{N\to\infty}F(\mu^N_t)-F(\mu^N_s)\\
        =&\lim_{N\to\infty}\int_s^t \int_\mathbb{R} L\frac{\delta}{\delta \mu} F(
        \mu^N_u,\cdot)(x,\mu_u,\alpha_u)\mu^N_u(dx)du\\
        +&\frac{1}{2}\int_s^t\int_\mathbb{R} \gamma(x,\mu_u,\alpha_u) \frac{\delta^2}{\delta \mu^2}F(\mu^N_u,x,x)\mu^N_u(dx)du\\
        +&\int_s^t\int_\mathbb{R}\sqrt{\gamma(x,\mu_u,\alpha_u(x))}\frac{\delta}{\delta \mu}F(\mu^N_u,x)\rho^N(x)\mathcal{M}(dx,du)\\
        =&\int_s^t \int_\mathbb{R} L\frac{\delta}{\delta \mu} F(
        \mu_u,\cdot)(x,\mu_u,\alpha_u)\mu_u(dx)du\\
        +&\frac{1}{2}\int_s^t\int_\mathbb{R} \gamma(x,\mu_u,\alpha_u(x)) \frac{\delta^2}{\delta \mu^2}F(\mu_u,x,x)\mu_u(dx)du\\
        +&\int_s^t\int_\mathbb{R}\sqrt{\gamma(x,\mu_u,\alpha_u(x))}\frac{\delta}{\delta \mu}F(\mu_u,x)\mathcal{M}(dx,du),
    \end{aligned}
\end{equation}
    where all the limits are in the sense of probability.
\end{proof}

\subsection{Some corollaries}
From the proof of Theorem \ref{theorem1}, we can expect that some similar results holds true when $(\mu_t)_{t\leq 0}$ is not a controlled superprocess. More specifically, we don't require $(\mu_t)_{t\leq 0}$ to satisfy Equation \ref{equation1} any more, but require it to have some decomposition like the decomposition of the semi-martingales. 

\begin{assumption}\label{assumption2}
    Suppose $(\Omega, \mathscr{F}, \mathbb{P}, (\mathscr{F}_t)_{t\geq 0})$ is a filtered probability space. 
    $$(\mu_t, \lambda^1_t, \lambda^2_t):[0,T]\times \Omega\to (M(\mathbb{R}))^3$$  
    are measure-valued processes adapted to $(\mathscr{F}_t)_{t\geq 0}$.
    $$\mathcal{M}: \Omega \times [0,T] \times \mathscr{B}(\mathbb{R})\to \mathbb{R}$$ 
is a $(\mathscr{F}_t)_{t\geq 0}$ orthogonal martingale measure with density measure $\lambda^2_u(dx)du$. And we have
\begin{equation}
\begin{aligned}
    \left<f, \mu_t -\mu_s \right>=\int_s^t\int_\mathbb{R} L_1 f(x,\mu_u,\omega
)\mu_u(dx)du\\
    +\int_s^t L_2 f(x)\mathcal{M}(dx,du).
\end{aligned}
\end{equation}
    $L_1$ and $L_2$ are stochastic second order differential operators, i.e., they has the form such that
    $$L_i f(x, \mu_t, \omega)= \sum_{k=0}^2 b_{ik}(x, \mu_t, \omega) f^{(k)}(x), $$
    where $b_{ik}$ are bounded and Lipschitz parameters.
\end{assumption}

\begin{corollary}\label{corollary1}
    Under Assumption \ref{assumption2}, for any $F\in C^{2,2}_b(M(\mathbb{R}))$ and $0\leq s<t $, we have
    \begin{equation}
    \begin{aligned}
          F(\mu_t)-F(\mu_s)=\int_s^t \int_\mathbb{R} L_1\frac{\delta}{\delta \mu} F(
        \mu_u,\cdot)(x,\mu_u)\mu_u(dx)du\\
        +\frac{1}{2}\int_s^t\int_\mathbb{R}  G(\mu_u,x,x)\mu_u(dx)du\\
        +\int_s^t\int_\mathbb{R}L_2\frac{\delta}{\delta \mu}F(\mu_u, \cdot)(\mu_u,x)\mathcal{M}(dx,du),
    \end{aligned}
    \end{equation}
    where $G(\mu,x,y)=L_2(G_1(\mu,x,\cdot))(\mu,y)$, and $G_1(\mu, x, y)=L_2\frac{\delta}{\delta \mu}F(\mu, \cdot, y)(\mu,x)$.
\end{corollary}

\begin{remark}
    We may expect some same results holds true when $L_i$ are not of second order or not even differential operators as long as $F$ are sufficiently smooth.
\end{remark}

    The corollary below extends Theorem \ref{theorem1} to time-dependent case.
\begin{corollary}\label{corollary2}
    Suppose $(\mu_t)_{t\in[0,T]}$ is a solution of Equation \ref{equation1}, for any $F\in C_b^{1,(2,2)}$, we have
     \begin{equation}
    \begin{aligned}
          F(t, \mu_t)-F(s, \mu_s) &=\int_s^t\frac{\partial}{\partial t}F(u,\mu_u)du  +\int_s^t \int_\mathbb{R} L\frac{\delta}{\delta \mu} F(u, 
        \mu_u,\cdot)(x,\mu_u,\alpha_u)\mu_u(dx)du\\
        &+\frac{1}{2}\int_s^t\int_\mathbb{R} \gamma(x,\mu_u,\alpha_u(x)) \frac{\delta^2}{\delta \mu^2}F(u, \mu_u,x,x)\mu_u(dx)du\\
        &+\int_s^t\int_\mathbb{R}\sqrt{\gamma(x,\mu_u,\alpha_u(x))}\frac{\delta}{\delta \mu}F(u, \mu_u,x)\mathcal{M}(dx,du).
    \end{aligned}
    \end{equation}
\end{corollary}

    To extend Theorem \ref{theorem1} to measures on $\mathbb{R}^d$, we only need to extend Proposition \ref{proposition1} and Proposition \ref{proposition2} to the d-dimensional case.
\begin{proposition}
    Let $K$ is a compact subset of $\mathbb{R}^{2d}$. Suppose $f:[0,T]\times K\to\mathbb{R}$ is continuous, and $f(u,\cdot,\cdot)$ is symmetric for any $u\in [0,T]$. Then there exists a sequence of polynomials $f_n:[0,T]\times\mathbb{R}\times\mathbb{R}\to\mathbb{R}$ such that $f_n(u,\cdot,\cdot)$ is symmetric for any $u\in [0,T]$ and $f_n$ converges to $f$ uniformly in $[0,T]\times K$.
\end{proposition}
\begin{proof}
    The proof is the same as Proposition \ref{proposition1}.
\end{proof}
\begin{proposition}
    Let $f:[0,T]\times\mathbb{R}^d\times\mathbb{R}^d\to\mathbb{R}$ be a polynomial such that $f(u,\cdot,\cdot)$ is symmetric for any $u\in [0,T]$. Then there exists some $m\in\mathbb{N}$ and polynomials $\lambda_j: [0,T]\to\mathbb{R}$ and $g_j:\mathbb{R}^d\to\mathbb{R}$, $j=1,2, \cdots,  m$, such that $$f(u,x,y)=\sum_{j=1}^m \lambda_j (u)g_j(x)g_j(y).$$
\end{proposition}
\begin{proof}
    Notice that 
    \begin{equation}\label{equation18}
        \begin{aligned}
            \prod_{i=1}^d x_i^{r_i}y_i^{s_i}+\prod_{i=1}^d x_i^{s_i}y_i^{r_i} &=(\prod_{i=1}^d x_i^{r_i}+ \prod_{i=1}^d x_i^{s_i})(\prod_{i=1}^d y_i^{r_i}+ \prod_{i=1}^d y_i^{s_i})\\
            &-\prod_{i=1}^d x_i^{r_i}\prod_{i=1}^d y_i^{r_i}-\prod_{i=1}^d x_i^{s_i}\prod_{i=1}^d y_i^{s_i}.
        \end{aligned}
    \end{equation}
    Since $f$ is symmetric, the term $\prod_{i=1}^d x_i^{r_i}y_i^{s_i}$ and the term $\prod_{i=1}^d x_i^{s_i}y_i^{r_i}$ has the same coefficient $\lambda(u)$. Adding all terms together and using Equation \ref{equation18} we get
    $$f(u,x,y)=\sum_{j=1}^m \lambda_j (u)g_j(x)g_j(y),$$
    for some proper $m$, $\lambda_j$ and $g_j$.
\end{proof}

\begin{corollary}
    Theorem \ref{theorem1}, Corollary \ref{corollary1} and Corollary \ref{corollary2} holds true when $\mu_t$ take valued in the space of finite measures on $\mathbb{R}^d $.
\end{corollary}

\section{HJB equation and verification theorem for controlled superprocesses}\label{sec4}

 Corollary \ref{corollary2} together with the dynamic principle for controlled superprocesses allow us to deduce the HJB eqauation and the verification theorem. The proof of dynamic principle for controlled superprocesses follows the path in \cite{bib2}. That is verifying the backward and forward stability of the family of sets of probability $\mathscr{P}_{t,\lambda}$, verifying the measurability of the domain the cost function and then using the measurable selection theorem. 

\begin{theorem}[Dynamic principle for controlled superprocesses]\label{theorem2}
    Given any $$(t,\lambda)\in [0,T]\times D([0,T];M(\mathbb{R}))$$ and $\tau$ a stopping time in $[0,T]$,  we have 
    $$V(t,\lambda)=\inf_{\mathbb{P}\in\mathscr{P}_{t,\lambda}}E^{\mathbb{P}}[\int_t^\tau f(x,\mu_u,a)m(da,dx,du)+V(\tau,\mu_{\tau\wedge\cdot})].$$
\end{theorem}
\begin{proof}
    We first verify that $(\mathscr{P}_{t,\lambda})_{(t,\lambda)\in\mathbb{R}^+\times D([0,T];M(\mathbb{R}))}$ satisfies backward condition defined in Assumption 2.2 in \cite{bib2}.
    Note that the space $\Omega$ is a Polish space with some suitable topology where the $\sigma$-algebra generated by its Borel sets is $\mathscr{F}_T$. Also note that for any stopping time $\tau\in [0,T]$, $\mathscr{F}_\tau$ is generated by the map $f:\Omega\to\Omega$ where the image space is equipped with the countably generated Borel $\sigma$-algebra and
    $$f(\omega)=f((\mu(\omega),m(\omega)))= (\mu_{\tau(\omega)\wedge\cdot},m^{\tau(\omega)}),$$
    where $m^t(B)=m(B\cap [0,t]\times \mathbb{R}\times A)$ for any $B\in \mathscr{B}([0,t]\times \mathbb{R}\times A)$. Therefore, $\mathscr{F}_\tau$ is countably generated. By Theorem 1.1.6, Theorem 1.1.8 in \cite{bib22}, given any $\mathbb{P}\in\mathscr{P}_{t,\lambda}$, there exists a family of regular conditional probability distribution $(\mathbb{P}_{w})_{w\in\Omega}$ with respect to $\mathscr{F}_\tau$. 
    Since $\{\omega:\mu_s(\omega)=\mu_s(w), 0\leq s\leq \tau(\omega)\}$ include the $\omega_0$-atom with respect to $\mathscr{F}_\tau$, we have $\mathbb{P}_{w}(\{\mu_s(\omega)=\mu_s(w), 0\leq s\leq t\})=1$ for $\mathbb{P}$-a.e $w$. Since
    $$\mathbb{P}(\{\omega: \exists \alpha\in B([0,T]\times\mathbb{R}; A) , m(ds, dx, da)=ds\mu_s(dx)\delta_{\alpha(s,x)}(da)\})=1,$$
    we have $$\mathbb{P}_{w}(\{\omega: \exists \alpha\in B([0,T]\times\mathbb{R}; A) , m(ds, dx, da)=ds\mu_s(dx)\delta_{\alpha(s,x)}(da)\})=1$$ for $\mathbb{P}$-a.e. $w$.
    Then for any fixed $F,\varphi\in C^2_b(\mathbb{R})$, by Theorem 1.2.10 in \cite{bib22},
    $$F_{\varphi}(\mu_s)-\int_t^s \int_{\mathbb{R}\times A} \mathscr{L}F_\varphi(x,\mu_u,a)m(da, dx, du)$$
    is a $(\mathbb{P}_{w},\mathscr{F})$ martingale for all $\tau(w)\leq t\leq s \leq T$ for $\mathbb{P}$-a.e. $w$. Notice that $C^2_b(\mathbb{R})$ has a countable dense subset $\mathscr{C}$ and $\mathscr{L}$ is continuous with respect to $F$ and $\varphi$. So we can find a $\mathbb{P}$-null set $N$, for any $w$ outside of $N$, any $F,\varphi\in C^2_b(\mathbb{R})$ and any stopping time $\tau_1\geq \tau(w)$, by bounded convergence theorem,
    \begin{equation}
        \begin{aligned}
            &E^{\mathbb{P}_{w}}[F_{\varphi}(\mu_{\tau_1})-\int_t^{\tau_1} \int_{\mathbb{R}\times A} \mathscr{L}F_\varphi(x,\mu_u,a)m(da, dx, du)]\\
            =&\lim_{n\to\infty}E^{\mathbb{P}_{\omega_0}}[F_{n,\varphi_n}(\mu_{\tau_1})-\int_t^{\tau_1} \int_{\mathbb{R}\times A} \mathscr{L}F_{n,\varphi_n}(x,\mu_u,a)m(da, dx, du)]\\
            =&0.
        \end{aligned}
    \end{equation}
    Therefore, $\mathbb{P}_{w}\in\mathscr{P}_{t,\lambda}$ for $\mathbb{P}$-a.e. $w$. Next we verify that $\mathscr{P}_{t,\lambda}$ satisfies the forward condition in Assumption 2.2 in \cite{bib2}.

    Given a probability measure $\mathbb{P}\in\mathscr{P}_{t,\lambda}$ and a family of probability kernels $(Q_{w})_{w\in\Omega}$ such that $Q_{w}\in\mathscr{P}_{\tau(w), \mu(w)}$ for $\mathbb{P}$-a.e. $w$. Fix some $F,\varphi\in \mathscr{C}$, for $\mathbb{P}$-a.e. $w$, $$F_{\varphi}(\mu_s)-\int_{\tau(w)}^s \int_{\mathbb{R}\times A} \mathscr{L}F_\varphi(x,\mu_u,a)m(da, dx, du)$$ is a $\delta_w \otimes_{\tau(w)} \mathbb{Q}_\cdot$- martingale for $s\in [\tau(w),T]$.
    By Theoreom 1.2.10 in \cite{bib22}, 
    $$F_{\varphi}(\mu_s)-\int_t^s \int_{\mathbb{R}\times A} \mathscr{L}F_\varphi(x,\mu_u,a)m(da, dx, du)$$
    is a  $\mathbb{P}\otimes_\tau \mathbb{Q}_\cdot$-martingale. It's easy to check that 
    $$\mathbb{P}\otimes_\tau \mathbb{Q}_\cdot (\{\omega: \exists \alpha\in B([0,T]\times\mathbb{R}; A) , m(ds, dx, da)=ds\mu_s(dx)\delta_{\alpha(s,x)}(da)\})=1$$
    and $\mathbb{P}\otimes_\tau \mathbb{Q}_\cdot (\{\mu_s(\omega)=\mu_s(\lambda), 0\leq s\leq t\}=1)$.Therefore $\mathbb{P}\otimes_{\tau}\mathbb{Q}_\cdot \in \mathscr{P}_{t,\lambda}$.  
    
    Now we consider the $\mathscr{F}_\tau \otimes \mathscr{B}(\mathscr{P}(\Omega))$-measurability of the set
    $$B:=\{(w,\mathbb{P})\in \Omega\times \mathscr{P}(\Omega)): \mathbb{P}\in\mathscr{P}_{\tau(w),\mu(w)} \}.$$
    Let 
    $$B_0:=\{(w,\mathbb{P}):\mathbb{P}(\{\omega:\mu_s(\omega)=\mu_s(w), 0\leq s\leq \tau(w)\})=1\},$$
    $$B_1:=\{(w,\mathbb{P}):\mathbb{P}(\{\omega: \exists \alpha\in B([0,T]\times\mathbb{R}; A) , m(ds, dx, da)=ds\mu_s(dx)\delta_{\alpha(s,x)}(da)\})=1\}$$
    and
    \begin{equation}
        \begin{aligned}
            B_{F,\varphi,s,t,c}:=\{(w,\mathbb{P}):E^\mathbb{P}[F_{\varphi}(\mu_s)-\int_0^s \int_{\mathbb{R}\times A} \mathscr{L}F_\varphi(x,\mu_u,a)m(da, dx, du)\\
            -([F_{\varphi}(\mu_t)-\int_0^t \int_{\mathbb{R}\times A} \mathscr{L}F_\varphi(x,\mu_u,a)m(da, dx, du))1_{C}]=0\}
        \end{aligned}
    \end{equation}
    for any $s,t\in [0,T]\cap \mathbb{Q}$, $F,\varphi\in\mathscr{C}$, $B\in \mathscr{C}(\mathscr{F}_s)$, where $\mathscr{C}(\mathscr{F}_s)$ is a countable algebra that generates $\mathscr{F}_s$. Notice that $B_0,B_1,B_{F,\varphi,s,t,C}$ are all $\mathscr{F}_\tau \otimes \mathscr{B}(\mathscr{P}(\Omega))$-measurable and 
    $$B=B_0\cap B_1 \cap \bigcap_{F,\varphi,s,t,C}B_{F,\varphi,s,t,C}.$$
    Hence $B$ is $\mathscr{F}_\tau \otimes \mathscr{B}(\mathscr{P}(\Omega))$-measurable.

    Since $J(\tau(w),\mu_{\tau(w)\wedge\cdot}(w),\mathbb{P})$ is $\mathscr{F}_\tau \otimes \mathscr{B}(\mathscr{P}(\Omega))$-measurable, by measurable selection theorem, there is a family of $\mathscr{F}^U_\tau$-kernel $(\tilde{\mathbb{Q}}^\varepsilon_w)_{w\in\Omega}$ such that $(w,\tilde{\mathbb{Q}}^\varepsilon_w)\in B$ and $$J(\tau(w),\mu_{\tau(w)\wedge\cdot}(w),\tilde{\mathbb{Q}}^\varepsilon_w)\leq V(\tau(w),\mu_{\tau\wedge\cdot}(w))+\varepsilon.$$
    For any given $\mathbb{P}\in \mathscr{P}_{t,\lambda}$ and $\tau\in [t,T]$, we can find a family of $\mathscr{F}_\tau$-measurable probability kernel $(\mathbb{Q}^\varepsilon_w)_{w\in\Omega}$ such that 
    \begin{itemize}
        \item $\mathbb{Q}^\varepsilon_\cdot$ is a $\mathbb{P}$-modification of $\tilde{\mathbb{Q}}^\varepsilon_\cdot$;
        \item $\mathbb{Q}^\varepsilon_w\in \mathscr{P}_{\tau(w), \mu(w)}$ for $\mathbb{P}$-a.e. $w$;
        \item $J(\tau(w),\mu_{\tau(w)\wedge\cdot}(w),\mathbb{Q}^\varepsilon_w)=J(\tau(w),\mu_{\tau(w)\wedge\cdot}(w),\tilde{\mathbb{Q}}^\varepsilon_w)$ for $\mathbb{P}$-a.e. $w$.
    \end{itemize}
    
    With the discussion above, the dynamic principle for controlled superprocesses follows from Theorem 2.1 in \cite{bib2}.
    \end{proof}

    \begin{remark}\label{remark2}
    For any two paths $\lambda^1$ and $\lambda^2$ such that $\lambda^1_t=\lambda^2_t$, any $\mathbb{P}^1\in\mathscr{P}_{t,\lambda^1}$ and $\mathbb{P}^2\in\mathscr{P}_{t,\lambda^2}$, suppose $(\mathbb{P}^1_w)_{w\in\Omega}$ (resp. $(\mathbb{P}^2_w)_{w\in\Omega}$) is a family of regular conditional probability distribution of $\mathbb{P}^1$ (resp. $\mathbb{P}^2$) with respect to $\mathscr{F}_t$. We have $\mathbb{P}^1 \otimes_t \mathbb{P}^2_\cdot \in\mathscr{P}_{t,\lambda^1}$ and $J(t,\lambda^1,\mathbb{P}^1\otimes_t \mathbb{P}^2_\cdot)=J(t,\lambda^2,\mathbb{P}^2)$. So $V(t,\lambda^1)=V(t,\lambda^2)$. Therefore, we can regard $V$ as a function defined in $[0,T]\times M(\mathbb{R})$ such that $V(t,\lambda_t)= V(t,\lambda)$ with abuse of notation. 
    \end{remark}
    
    Next we give a formal derivation of HJB equation for controlled superprocesses. With Remark \ref{remark2}, we can consider the differential property of $V$. And we will use $\lambda$ to denote a measure instead of a measure-valued process for simplicity.
    \begin{theorem}\label{theorem3}
        If $V(t,\lambda)\in C^{1,(2,2)}_b$, then $V$ is the solution of the following HJB equation.
        \begin{equation}\label{equation22}
            \left\{
                \begin{aligned}
                     &\frac{\partial}{\partial t}V(u,\lambda)+\int_\mathbb{R} \inf_{a\in A}[L\frac{\delta}{\delta \mu} V(u, \lambda,\cdot)(x,\lambda,a)\\
           &\qquad \qquad+\frac{1}{2}\gamma(x,\lambda,a) \frac{\delta^2}{\delta \mu^2}V(u, \lambda,x,x)+f(x,\lambda,a)]\lambda(dx)= 0 \\
                     & V(T,\lambda)=g(T,\lambda)
                \end{aligned}
            \right.
        \end{equation}  
    \end{theorem}
\begin{proof}
    Given any $a_0:\mathbb{R}\to A$ and $s\in[t,T]$, there is some $\mathbb{P}\in\mathscr{P}_{t,\lambda}$ such that $\mathbb{P}(m(du, dx,da )=du\mu_u(dx)\delta_{a_0(x)}(da))=1$. By Theorem \ref{theorem2}, 
    \begin{equation}\label{equation19}
        \begin{aligned}
            V(t,\lambda)&\leq E^\mathbb{P}[\int_t^s f(x,\mu_u,a)m(da,dx,du)+V(s,\mu_{s})]\\
            &= E^\mathbb{P}[\int_t^s f(x,\mu_u,a_0(x))\mu_u(dx)du+V(s,\mu_{s})].
        \end{aligned}
    \end{equation}
    By Remark\ref{remark1}, Proposition\ref{proposition3} and Corollary\ref{corollary2},
    \begin{equation}\label{equation20}
    \begin{aligned}
        E^\mathbb{P}[V(s,\mu_s)-V(t,\lambda)]&=E^\mathbb{P}[\int_t^s\frac{\partial}{\partial t}V(u,\mu_u)du\\
        &+\int_t^s \int_\mathbb{R} L\frac{\delta}{\delta \mu} V(u, 
        \mu_u,\cdot)(x,\mu_u,a_0(x))\mu_u(dx)du\\
        &+\frac{1}{2}\int_t^s\int_\mathbb{R} \gamma(x,\mu_u,a_0(x)) \frac{\delta^2}{\delta \mu^2}V(u, \mu_u,x,x)\mu_u(dx)du]
    \end{aligned}
    \end{equation}
    Combining \ref{equation19} and \ref{equation20}, dividing both sides by $t-s$ then let $s\to t$, we get
    
    \begin{equation}
        \begin{aligned}
            &\frac{\partial}{\partial t}V(u,\lambda)+\int_\mathbb{R} [L\frac{\delta}{\delta \mu} V(u, \lambda,\cdot)(x,\lambda,a_0(x))\\
            +&\frac{1}{2}\gamma(x,\lambda,a_0(x)) \frac{\delta^2}{\delta \mu^2}V(u, \lambda,x,x)+f(x,\lambda,a_0(x))]\lambda(dx)\geq 0.
        \end{aligned}
    \end{equation}
    Therefore, 
        \begin{equation}
        \begin{aligned}
            &\frac{\partial}{\partial t}V(u,\lambda)+\int_\mathbb{R} \inf_{a\in A}[L\frac{\delta}{\delta \mu} V(u, \lambda,\cdot)(x,\lambda,a)\\
            +&\frac{1}{2}\gamma(x,\lambda,a) \frac{\delta^2}{\delta \mu^2}V(u, \lambda,x,x)+f(x,\lambda,a)]\lambda(dx)\geq 0.
        \end{aligned}
    \end{equation}
    On the other hand, given any $\varepsilon>0$ and $s>t$, there is some $\mathbb{P}\in\mathscr{P}_{t,\lambda}$ such that $\mathbb{P}(\{(\mu,m):\exists\alpha, m(du,dx,da )=du\mu_u(dx)\delta_{\alpha_u(x)}(da)\})=1$ and
    \begin{equation}
        \begin{aligned}\label{equation21}
             V(t,\lambda)+\varepsilon (t-s)&\geq E^\mathbb{P}[\int_t^s f(x,\mu_u,a)m(da,dx,du)+V(s,\mu_{s})]\\
            &= E^\mathbb{P}[\int_t^s f(x,\mu_u,\alpha_u(x))\mu_u(dx)du+V(s,\mu_{s})].
        \end{aligned}
    \end{equation}
    By Remark\ref{remark1}, Proposition\ref{proposition3} and Corollary\ref{corollary2},
    \begin{equation}\label{equation23}
        \begin{aligned}
        E^\mathbb{P}[V(s,\mu_s)-V(t,\lambda)]&=E^\mathbb{P}[\int_t^s\frac{\partial}{\partial t}V(u,\mu_u)du\\
        &+\int_t^s \int_\mathbb{R} L\frac{\delta}{\delta \mu} V(u, 
        \mu_u,\cdot)(x,\mu_u,\alpha_u(x))\mu_u(dx)du\\
        &+\frac{1}{2}\int_t^s\int_\mathbb{R} \gamma(x,\mu_u,\alpha_u(x)) \frac{\delta^2}{\delta \mu^2}V(u, \mu_u,x,x)\mu_u(dx)du]
    \end{aligned}
    \end{equation}
    Combining \ref{equation23} and \ref{equation21}, dividing both sides by $t-s$, we get
    \begin{equation}
        \begin{aligned}
            0&\geq \lim_{s\to t}E^{\mathbb{P}}[\frac{1}{s-t}\int_t^s(\frac{\partial}{\partial t}V(u,\mu_u)+\int_\mathbb{R}(L\frac{\delta}{\delta \mu} V(u, 
        \mu_u,\cdot)(x,\mu_u,\alpha_u(x))\mu_u(dx)du\\
        &\qquad\qquad+\frac{1}{2}\gamma(x,\mu_u,\alpha_u(x)) \frac{\delta^2}{\delta \mu^2}V(u, \mu_u,x,x)+f(x,\mu_u,\alpha_u(x)))\mu_u(dx))du)]\\
            &\geq \lim_{s\to t} E^{\mathbb{P}}[\frac{1}{s-t}\int_t^s(\frac{\partial}{\partial t}V(u,\mu_u)+\int_\mathbb{R}\inf_{a\in A}(L\frac{\delta}{\delta \mu} V(u, 
        \mu_u,\cdot)(x,\mu_u,a)\mu_u(dx)du\\
        &\qquad\qquad\qquad\qquad+\frac{1}{2}\gamma(x,\mu_u,a) \frac{\delta^2}{\delta \mu^2}V(u, \mu_u,x,x)+f(x,\mu_u,a))\mu_u(dx))du)]\\
        &=\frac{\partial}{\partial t}V(u,\lambda)+\int_\mathbb{R} \inf_{a\in A}[L\frac{\delta}{\delta \mu} V(u, \lambda,\cdot)(x,\lambda,a))\\
            &\qquad\qquad\qquad\qquad\qquad+\frac{1}{2}\gamma(x,\lambda,a) \frac{\delta^2}{\delta \mu^2}V(u, \lambda,x,x)+f(x,\lambda,a)]\lambda(dx).
        \end{aligned}
    \end{equation}
\end{proof}

The following theorem is the inverse of Theorem \ref{theorem3} with the smooth condition.

    \begin{theorem}[Verification theorem for controlled superprocesses]\label{theorem4}
        Given $$W\in C^{1,(2,2)}_b([0,T)\times M(\mathbb{R}))\cap C^0([0,T]\times M(\mathbb{R}))$$.
        \begin{itemize}
            \item [(i)]
            Suppose that
            \begin{equation}
                \left\{
                \begin{aligned}
                     &\frac{\partial}{\partial t}W(u,\lambda)+\int_\mathbb{R} \inf_{a\in A}[L\frac{\delta}{\delta \mu} W(u, \lambda,\cdot)(x,\lambda,a)\\
           &\qquad \qquad+\frac{1}{2}\gamma(x,\lambda,a) \frac{\delta^2}{\delta \mu^2}W(u, \lambda,x,x)+f(x,\lambda,a)]\lambda(dx)\geq 0 \\
                     & W(T,\lambda)\leq g(T,\lambda).
                \end{aligned}
            \right.
            \end{equation}
            Then $W\leq V$ on $[0,T]\times M(\mathbb{R})$.
            \item [(ii)]
            Suppose further that $W(T,\lambda)=g(\lambda)$ and there exists a measurable function $\alpha^*:[0,T]\times M(\mathbb{R})\times\mathbb{R}\to\mathbb{R}$ such that
            \begin{equation}\label{equation24}
                \begin{aligned}
                    &L\frac{\delta}{\delta \mu} W(u, \lambda,\cdot)(x,\lambda,\alpha^*(u,\lambda,x))\\
                    &\qquad\qquad+\frac{1}{2}\gamma(x,\lambda,\alpha^*(u,\lambda,x)) \frac{\delta^2}{\delta \mu^2}W(u, \lambda,x,x)+f(x,\lambda,\alpha^*(u,\lambda,x))\\
                    =&\inf_{a\in A}[L\frac{\delta}{\delta \mu} W(u, \lambda,\cdot)(x,\lambda,a)\\
           &\qquad \qquad+\frac{1}{2}\gamma(x,\lambda,a) \frac{\delta^2}{\delta \mu^2}W(u, \lambda,x,x)+f(x,\lambda,a)],
                \end{aligned}
            \end{equation}
            and for any $(t,\lambda)$, there exists some $\mathbb{P}\in\mathscr{P}_{t,\lambda}$ such that
            $$\mathbb{P}(\{(\mu,m):m(du,dx,da)=du\mu_u(dx)\delta_{\alpha^*(u,\mu_u,x)}(da)\})=1.$$
            Then $W=V$ and $\alpha^*$ is an optimal Markovian control.
        \end{itemize}
    \end{theorem}
    \begin{proof}
        \begin{itemize}
            \item [(i)]
                Given any $(t,\lambda)$ and $\mathbb{P}\in\mathscr{P}_{t,\lambda}$ such that $$\mathbb{P}(\{(\mu,m):\exists\alpha, m(du,dx,da )=du\mu_u(dx)\delta_{\alpha_u(x)}(da)\})=1,$$
                 by Remark\ref{remark1}, Proposition\ref{proposition3} and Corollary\ref{corollary2},
                \begin{equation}\label{equation25}
                    \begin{aligned}
                     E^\mathbb{P}[W(s,\mu_s)-W(t,\lambda)]&=E^\mathbb{P}[\int_t^s\frac{\partial}{\partial t}W(u,\mu_u)du\\
                     &+\int_t^s \int_\mathbb{R} L\frac{\delta}{\delta \mu} W(u, 
                    \mu_u,\cdot)(x,\mu_u,\alpha_u(x))\mu_u(dx)du\\
                    &+\frac{1}{2}\int_t^s\int_\mathbb{R} \gamma(x,\mu_u,\alpha_u(x)) \frac{\delta^2}{\delta \mu^2}W(u, \mu_u,x,x)\mu_u(dx)du]
                    \end{aligned}
                \end{equation}
                Combining \ref{equation24} and \ref{equation25}, we have
                \begin{equation}\label{equation26}
                    \begin{aligned}
                        E^\mathbb{P}[W(s,\mu_s)-W(t,\lambda)]&\geq E^\mathbb{P}[\int_t^s -f(x,\mu_u,\alpha_u(x))\mu_u(dx)du] 
                    \end{aligned}
                \end{equation}
                Let $s=T$ we get $W\leq V$.
            \item [(ii)]
                With additional condition, \ref{equation26} transforms into 
                \begin{equation}
                    E^\mathbb{P}[W(s,\mu_s)-W(t,\lambda)]\geq E^\mathbb{P}[\int_t^s -f(x,\mu_u,\alpha^*(u,\mu_u,x))\mu_u(dx)du].
                \end{equation}
                Hence $W=V$.
        \end{itemize}
    \end{proof}

\section{Viscosity solutions for some measure-valued control problem}\label{sec5}

In this section, we give a definition of the viscosity solution of the HJB equation evolving derivatives of finite measures. Then we prove that the value function is the viscosity solution and the uniqueness of this type of viscosity solution in a special case when there is no second partial derivative with respect to the measure, that is $\gamma\equiv 0$ in our setting.

First we give another definition of the linear derivative.

\begin{definition}\label{definition4}
    For a function $f: M(\mathbb{R})\to\mathbb{R}$, the linear derivative of $f$ at $\mu$ (if exists) is a continuous and bounded function $\frac{\partial}{\partial \mu} f: \mathbb{R}\to\mathbb{R}$ such that
    \begin{equation}
        f(\mu+m)=f(\mu)+\int_\mathbb{R}\frac{\partial}{\partial \mu}f(x)\mu(dx)+o(\|m\|),
    \end{equation}
    where $\|m\|:=d(m,0)$.
    Moreover, if the linear derivative of $f$ exists at each point in $M(\mathbb{R})$, we denote $f\in C^1_b(M(\mathbb{R}))$ and we use $\frac{\partial}{\partial\mu}f(\mu,\cdot)$ to denote its linear derivative at $\mu$.
\end{definition}
\begin{remark}
    Compared with the linear derivative of $L^2$-probability measures, we can't introduce a distance that resembles the Wasserstein distance defined on $L^2$-probability measures and using the optimal transport technique to prove the equivalence between Definition \ref{definition3} and Definition \ref{definition4}. But it's obvious that Definition \ref{definition4} is more strict.
\end{remark}
The Hamiltonian for our stochastic problem is a function
$$H:M(\mathbb{R})\times C^1_b(\mathbb{R})\times C_b(\mathbb{R})\times C^2_b(\mathbb{R}^2)\to\mathbb{R}$$ 
defined as
\begin{equation}
    \begin{aligned}
        H(\mu,p,M,r)=\int_{\mathbb{R}}\inf_{a\in A}[b(x,\mu,a)p(x)+\frac{1}{2}\sigma(x,\mu,a)^2M(x)\\
        +\frac{1}{2}\gamma(x,\mu,a)r(x,x)+f(x,\mu,a)]\mu(dx).
    \end{aligned}
\end{equation}

\begin{definition}
    A continuous function $U:[0,T]\times M(\mathbb{R})$ is called a viscosity subsolution (respectively supersolution) of 
    \begin{equation}\label{equation27}
    \begin{aligned}
        &-\frac{\partial}{\partial t}V(t,\mu)-\beta V(t,\lambda)\\
        &\qquad\qquad -H(\mu,\frac{\partial}{\partial x}\frac{\partial}{\partial\mu}V(t,\mu,x),\frac{\partial^2}{\partial x^2}\frac{\partial}{\partial\mu}V(t,\mu,x),\frac{\partial^2}{\partial\mu^2}V(t,\mu,x,x))=0
    \end{aligned}
    \end{equation}
    if 
    \begin{equation}
    \begin{aligned}
        &-\frac{\partial}{\partial t}\varphi(t,\mu)-\beta V(t,\lambda)\\
        &\qquad\qquad -H(\mu,\frac{\partial}{\partial x}\frac{\partial}{\partial\mu}\varphi(t,\mu,x),\frac{\partial^2}{\partial x^2}\frac{\partial}{\partial\mu}\varphi(t,\mu,x),\frac{\partial^2}{\partial\mu^2}\varphi(t,\mu,x,x))\leq 0
    \end{aligned}
    \end{equation}
    (respectively $\geq$) for any $(t,\mu)\in [0,T]\times M(\mathbb{R})$ and $\varphi\in C^{1,(2,2)}_b$ such that $(t,\mu)$ is a local maximum (respectively minimum) point of $U-\varphi$. And $U$ is called a viscostiy solution of Equation \ref{equation27} if it is both a viscosity subsolution and supersolution of Equation \ref{equation27}.
\end{definition}

With the definition of viscosity solutions above, the smooth assumption of Theorem \ref{theorem3} can be relaxed to the assumption of continuity and the viscosity version of Theorem \ref{theorem4} holds true when the term of the second order derivative of the measure vanishes.

\begin{theorem}
    If the value function $V$ is continuous, the $V$ is the viscosity solution of Equation \ref{equation27} with terminal condition $V(t,\mu)=g(\mu)$.
\end{theorem}
\begin{proof}
    The proof of this theorem is entirely analogous to the classical case.
\end{proof}

\begin{theorem}[Strong comparison theorem]
    When $\gamma\equiv 0$, for any viscosity subsolution $U$ (supersolution $W$) of Equation \ref{equation27} with polynomial growth condition and terminal condition $U(t,\mu)=W(t,\mu)=g(\mu)$, we have $U\leq W$. Therefore, $V$ is the only viscosity solution of Equation \ref{equation27} with terminal conditon $V(T,\mu)=g(\mu)$.
\end{theorem}
\begin{proof}
The proof follows the spirit of Theorem 4.4.3 and Theorem 4.4.4 in \cite{bib11},
In the case when $\gamma\equiv 0$, with some abuse of notation, we let the Hamiltonian 
$H$ be defined as
\begin{equation}
    \begin{aligned}
        H(\mu,p,M)=\int_{\mathbb{R}}\inf_{a\in A}[b(x,\mu,a)p(x)+\frac{1}{2}\sigma(x,\mu,a)^2M(x)+f(x,\mu,a)]\mu(dx).
    \end{aligned}
\end{equation}
     we first consider the HJB equation in the form
    \begin{equation}
    \begin{aligned}
        &-\frac{\partial}{\partial t}V(t,\mu)+\beta V(t,\mu)\\
        &\qquad -H(\mu,\frac{\partial}{\partial x}\frac{\partial}{\partial\mu}V(t,\mu,x),\frac{\partial^2}{\partial x^2}\frac{\partial}{\partial\mu}V(t,\mu,x))=0.
    \end{aligned}
    \end{equation}
\end{proof}
Note that Equation \ref{equation27} is a special case when $\beta=0$. However, we only need to consider the case when $\beta>0$ without loss of generality. Because if we let $\tilde{U}(t,\mu)=e^{\theta t}U(t,\mu)$ and $\tilde{W}(t,\mu)=e^{\theta t}W(t,\mu)$, then $\tilde{U}$ (respectively $\tilde{W}$) is the viscosity subsolution of 
\begin{equation}
\begin{aligned}
     &-\frac{\partial}{\partial t}V(t,\mu)+(\beta+\theta) V(t,\mu)\\
        &\qquad -\tilde{H}(\mu,\frac{\partial}{\partial x}\frac{\partial}{\partial\mu}V(t,\mu,x),\frac{\partial^2}{\partial x^2}\frac{\partial}{\partial\mu}V(t,\mu,x))=0,
\end{aligned}
\end{equation}
where $\tilde{H}$ is the Hamiltonian with $f(t,\mu,a)$ replaced by $e^{\theta t}f(t,\mu,a)$. Since we can take $\theta$ large enough so that $\beta+\theta>0$ and $U\leq W$ if and only if $\tilde{U}\leq \tilde{W}$, we can assume $\beta>0$ without loss of generality.

Since a bounded close set in $M(\mathbb{R})$ is not necessary compact, we use some localization method to overcome this difficulty. Let $U^N(t,\mu)=U(t,\mu^N)$, $W^N(t,\mu)=W(t,\mu^N)$,
\begin{equation}
\begin{aligned}
    &H^N(\mu,p,M)=\int_{\mathbb{R}}\inf_{a\in A}[b(x,\mu,a)p(x)+\frac{1}{2}\sigma(x,\mu,a)^2M(x)\\
    &\qquad\qquad\qquad\qquad\qquad+f(x,\mu,a)]\rho^N(x)\mu(dx).
\end{aligned}
\end{equation}
Then $U^N$ (respectively $W^N$) is the viscosity subsolution (respectively supersolution) of 
\begin{equation}\label{equation28}
    \begin{aligned}
     &-\frac{\partial}{\partial t}V(t,\mu)+\beta V(t,\mu)\\
        &\qquad -H^N(\mu,\frac{\partial}{\partial x}\frac{\partial}{\partial\mu}V(t,\mu,x),\frac{\partial^2}{\partial x^2}\frac{\partial}{\partial\mu}V(t,\mu,x))=0,
    \end{aligned}
\end{equation}
If we can prove $U^N\leq W^N$, we can obtain $U\leq W$ from the continuity assumption.
Since $U,W$ satisfies the polynomial growth condition, we can find a $p>1$ such that
\begin{equation}
    \sup_{[0,t]\times M(\mathbb{R})}\frac{U^N(t,\mu)+W^N(t,\mu)}{1+\|\mu^N\|^p}<\infty.
\end{equation}
Define $\phi(t,\mu)=e^{-\theta t}(1+\|\mu\|^{2p});=e^{-\theta t}\psi(\mu)$, it's obvious that $\phi$ is smooth. Notice that
\begin{equation}
    \begin{aligned}
     &-\frac{\partial}{\partial t}\phi(t,\mu)+\beta \phi(t,\mu)
        -H^N(\mu,\frac{\partial}{\partial x}\frac{\partial}{\partial\mu}\phi(t,\mu,x),\frac{\partial^2}{\partial x^2}\frac{\partial}{\partial\mu}\phi(t,\mu,x))\\
        =&e^{-\theta t}[(\beta+\theta)\psi(\mu)-H^N(\mu,\frac{\partial}{\partial x}\frac{\partial}{\partial\mu}\psi(\mu,x),\frac{\partial^2}{\partial x^2}\frac{\partial}{\partial\mu}\psi(\mu,x))]\\
    \end{aligned}
\end{equation}
By direct calculation we have
\begin{equation}
\begin{aligned}
    \int_{\mathbb{R}}\frac{\partial}{\partial x}\frac{\partial}{\partial \mu}\psi(\mu,x)\mu(dx)&=2p\|\mu\|^{2p-2}\sum_{k=0}^\infty\frac{1}{2^k q_k} \langle\varphi_k,\mu\rangle\langle\varphi_k',\mu\rangle \\
    &\leq 2p\|\mu\|^{2p-1}(\sum_{k=0}^\infty\frac{1}{2^k q_k} \langle\varphi_k',\mu\rangle)^{\frac{1}{2}}\\
    &\leq 2p\|\mu\|^{2p-1}(\sum_{k=0}^\infty\frac{1}{2^k} \langle 1,\mu\rangle)^{\frac{1}{2}}\\
    &=4p\|\mu\|^{2p-1}(\langle\varphi_0,\mu\rangle)^{\frac{1}{2}}\\
    &\leq 4p\psi.
\end{aligned}
\end{equation}
For the same reason $\int_{\mathbb{R}}\frac{\partial^2}{\partial x^2}\frac{\partial}{\partial \mu}\psi(\mu,x)\mu(dx)\leq 4p\psi$. Therefore there exists a constant $C$ such that 
\begin{equation}
\begin{aligned}
    &e^{-\theta t}[(\beta+\theta)\psi(\mu)-H^N(\mu,\frac{\partial}{\partial x}\frac{\partial}{\partial\mu}\psi(\mu,x),\frac{\partial^2}{\partial x^2}\frac{\partial}{\partial\mu}\psi(\mu,x))]\\
    \geq& (\beta+\theta-C)\phi(\mu).
\end{aligned}
\end{equation}
Then $\phi$ is a viscosity supersolution of Equation \ref{equation28} when choosing $\theta$ large enough. So for any $\varepsilon>0$, $W^N+\varepsilon \phi$ is a viscosity supersolution of Equation \ref{equation28}. If for any $\varepsilon$, $U^N-W^N-\varepsilon\phi\leq 0$, then the proof is done. Now suppose there exists some $\varepsilon_0$ such that $\sup \{U^N-W^N-\varepsilon_0\phi\}>0$. Define
$$K_{m,n}:=\{\mu\in M(\mathbb{R}):\|\mu\|\leq m,\mu([-n,n]^c)=0\}.$$
Then $K_{m,n}$ is compact with respect to the topology induced by $d$.
Since
\begin{equation}
    \lim_{\|\mu\|\to\infty}\sup_{[0,T]}\{U^N-W^N-\varepsilon_0\phi\}=-\infty,
\end{equation}
 $U^N(t,\mu)=U^N(t,\mu^N)$, $W^N(t,\mu)=W^N(t,\mu^N)$ and $\phi(t,\mu)\geq \phi(t,\mu^N)$,we have
 \begin{equation}
     \sup_{[0,T]\times M(E)}\{U^N-W^N-\varepsilon_0\phi\}=\sup_{[0,T]\times K_{m,n}}\{U^N-W^N-\varepsilon_0\phi\}>0
 \end{equation}
for some $m,n$.
Since $[0,T]\times K_{m,n}$ is compact, there is some $(\bar{t},\bar{\mu})\in [0,T]\times K{m,n}$ that reaches the supremum 
\begin{equation}\label{equation33}
    M:=\sup_{[0,T]\times M(E)}\{U^N-W^N-\varepsilon_0\phi\}>0.
\end{equation}
Let $\mathcal{O}$ be an open subset of $M(\mathbb{R})$ such that $K_{m,n}\subset \mathcal{O}\subset K_{m+1,\infty}$. Define a family of test functions 
$$\Phi_\varepsilon(t,s,\mu,\lambda)=U^N(t,\mu)-W^N(s,\lambda)-\varepsilon_0\phi(s,\lambda)-\phi_\varepsilon(t,s,\mu,\lambda),$$
where 
$$\phi_\varepsilon(t,s,\mu,\lambda)=\frac{1}{\varepsilon}(|t-s|^2+\|\mu-\lambda\|^2).$$
For the same reason we have
\begin{equation}
    M_\varepsilon:=\sup_{[0,T]^2\times \mathcal{O}^2}\Phi_\varepsilon=\sup_{[0,T]^2\times K_{m+1,n}^2}\Phi_\varepsilon.
\end{equation}
Hence there exists some $(t_\varepsilon,s_\varepsilon,\mu_\varepsilon,\lambda_\varepsilon)\in [0,T]^2\times K^2_{m+1,n}$ such that $$\Phi_\varepsilon(t_\varepsilon,s_\varepsilon,\mu_\varepsilon,\lambda_\varepsilon)=M_\varepsilon.$$
We get $M_\varepsilon\geq M$ by letting $s=t$, $\lambda=\mu$. Let $\varepsilon\to 0$, since $K^2_{m+1,n}$ is compact, we can find a subsequence of $(t_\varepsilon,s_\varepsilon,\mu_\varepsilon,\lambda_\varepsilon)$ converges to some $(t_0,s_0,\mu_0,\lambda_0)\in [0,T]^2\times K^2_{m+1,n}$. 
We must have $s_0=t_0$ and $\mu_0=\lambda_0$, otherwise $\phi_\varepsilon(t_\varepsilon,s_\varepsilon,\mu_\varepsilon,\lambda_\varepsilon)\to\infty$ as $\varepsilon\to 0$ but $U^N(t_\varepsilon,\mu_\varepsilon)-W^N(s_\varepsilon,\lambda_\varepsilon)-\varepsilon_0 \phi(s_\varepsilon,\lambda_\varepsilon)$ is bounded with respect to $\varepsilon$, which contradicts with the fact that $M_\varepsilon\geq M$. Therefore, $$M\leq \lim_{\varepsilon\to0}M_\varepsilon\leq M.$$ So $$\lim_{\varepsilon\to 0}\phi_\varepsilon(t_\varepsilon,s_\varepsilon,\mu_\varepsilon,\lambda_\varepsilon)=0.$$
Notice that $(t_\varepsilon,\mu_\varepsilon)$ is a local maximum of
$$U^N(t,\mu)-W^N(s_\varepsilon,\lambda_\varepsilon)-\varepsilon_0\phi(s_\varepsilon,\lambda_\varepsilon)-\phi_\varepsilon(t,s_\varepsilon,\mu,\lambda_\varepsilon).$$
Since $U^N$ is a subsolution of Equation \ref{equation28}, and $\phi(t,\mu)=\phi_\varepsilon(t,s_\varepsilon,\mu,\lambda_\varepsilon)$ is smooth, we obtain
\begin{equation}\label{equation29}
    \begin{aligned}
     &-\frac{2}{\varepsilon}(t_\varepsilon-s_\varepsilon)+\beta U^N(t_\varepsilon,\mu_\varepsilon)\\
        &\qquad -H^N(\mu_\varepsilon,\frac{2}{\varepsilon}\sum_{k=0}^\infty \frac{1}{2^k q_k}\langle \phi_k,\mu_\varepsilon-\lambda_\varepsilon\rangle\phi_k',\frac{2}{\varepsilon}\sum_{k=0}^\infty \frac{1}{2^k q_k}\langle \phi_k,\mu_\varepsilon-\lambda_\varepsilon\rangle\phi_k'')\leq 0.
    \end{aligned}
\end{equation}
Similarly, considering the viscosity supersolution $W^N-\varepsilon_0 \phi$, we have
\begin{equation}\label{equation30}
    \begin{aligned}
     &-\frac{2}{\varepsilon}(t_\varepsilon-s_\varepsilon)+\beta (W^N(s_\varepsilon,\lambda_\varepsilon)+\varepsilon_0 \phi(s_\varepsilon,\lambda_\varepsilon))\\
        &\qquad -H^N(\lambda_\varepsilon,\frac{2}{\varepsilon}\sum_{k=0}^\infty \frac{1}{2^k q_k}\langle \phi_k,\mu_\varepsilon-\lambda_\varepsilon\rangle\phi_k',\frac{2}{\varepsilon}\sum_{k=0}^\infty \frac{1}{2^k q_k}\langle \phi_k,\mu_\varepsilon-\lambda_\varepsilon\rangle\phi_k'')\leq 0.
    \end{aligned}
\end{equation}
Subtracting Equation \ref{equation29} by Equation \ref{equation30}, we get
\begin{equation}
    \begin{aligned}
        &\beta(U^N(t_\varepsilon,\mu_\varepsilon)-W^N(s_\varepsilon,\lambda_\varepsilon)-\varepsilon_0\phi(s_\varepsilon,\lambda_\varepsilon))\\
        \leq& H^N(\mu_\varepsilon,\frac{2}{\varepsilon}\sum_{k=0}^\infty \frac{1}{2^k q_k}\langle \phi_k,\mu_\varepsilon-\lambda_\varepsilon\rangle\phi_k',\frac{2}{\varepsilon}\sum_{k=0}^\infty \frac{1}{2^k q_k}\langle \phi_k,\mu_\varepsilon-\lambda_\varepsilon\rangle\phi_k'')\\
        &\qquad-H^N(\lambda_\varepsilon,\frac{2}{\varepsilon}\sum_{k=0}^\infty \frac{1}{2^k q_k}\langle \phi_k,\mu_\varepsilon-\lambda_\varepsilon\rangle\phi_k',\frac{2}{\varepsilon}\sum_{k=0}^\infty \frac{1}{2^k q_k}\langle \phi_k,\mu_\varepsilon-\lambda_\varepsilon\rangle\phi_k'').
    \end{aligned}
\end{equation}
We denote
$$p_\varepsilon(x):=\frac{2}{\varepsilon}\sum_{k=0}^\infty \frac{1}{2^k q_k}\langle \phi_k,\mu_\varepsilon-\lambda_\varepsilon\rangle\phi_k'(x)$$
and
$$M_\varepsilon(x):=\frac{2}{\varepsilon}\sum_{k=0}^\infty \frac{1}{2^k q_k}\langle \phi_k,\mu_\varepsilon-\lambda_\varepsilon\rangle\phi_k''(x)$$
for simplicity. By direct calculation we have
\begin{equation}\label{equation31}
\begin{aligned}
    &\left|\int_{\mathbb{R}}b(x,\mu_\varepsilon,a)p_\varepsilon(x)\mu_\varepsilon(dx)-\int_{\mathbb{R}}b(x,\lambda_\varepsilon,a)p_\varepsilon(x)\lambda_\varepsilon(dx)\right|\\
    \leq&\left|\int_{\mathbb{R}}b(x,\mu_\varepsilon,a)p_\varepsilon(x)\mu_\varepsilon(dx)-\int_{\mathbb{R}}b(x,\mu_\varepsilon,a)p_\varepsilon(x)\lambda_\varepsilon(dx)\right|\\
    &+\left|\int_{\mathbb{R}}b(x,\mu_\varepsilon,a)p_\varepsilon(x)\lambda_\varepsilon(dx)-\int_{\mathbb{R}}b(x,\lambda_\varepsilon,a)p_\varepsilon(x)\lambda_\varepsilon(dx)\right|\\
    \leq& |b(x,\mu_\varepsilon,a)| \left|\langle p_\varepsilon, \mu_\varepsilon-\lambda_\varepsilon\rangle\right|+\|\lambda_\varepsilon\|\cdot\|(b(x,\mu_\varepsilon,a)-b(x,\lambda_\varepsilon,a)p_\varepsilon(x)\|_\infty\\
    \leq& C_1 \left|\frac{2}{\varepsilon}\sum_{k=0}^\infty\frac{1}{2^k q_k}\langle\phi_k,\mu_\varepsilon-\lambda_\varepsilon\rangle\langle\phi_k',\mu_\varepsilon-\lambda_\varepsilon\rangle\right|\\
    &\qquad+C_2\|\mu_\varepsilon-\lambda_\varepsilon\|\frac{2}{\varepsilon}\sum_{k=0}^\infty\frac{1}{2^k}\left|\langle 1,\mu_\varepsilon-\lambda_\varepsilon\rangle\right|\\
    \leq & C_3\frac{2}{\varepsilon}\|\mu_\varepsilon-\lambda_\varepsilon\|^2 ,
\end{aligned}
\end{equation}
where $C_1, C_2, C_3$ are constants independent of $\varepsilon$. Similarly, we have
\begin{equation}\label{equation32}
    \left|\int_{\mathbb{R}}\sigma^2(x,\mu_\varepsilon,a)M_\varepsilon(x)\mu_\varepsilon(dx)-\int_{\mathbb{R}}\sigma^2(x,\lambda_\varepsilon,a)M_\varepsilon(x)\lambda_\varepsilon(dx)\right|\leq C_4\frac{2}{\varepsilon}\|\mu_\varepsilon-\lambda_\varepsilon\|^2
\end{equation}
for some constant $C_4$. Combining Equation \ref{equation31} and Equation \ref{equation32}, we have
\begin{equation}
    \begin{aligned}
        &\beta(U^N(t_\varepsilon,\mu_\varepsilon)-W^N(s_\varepsilon,\lambda_\varepsilon)-\varepsilon_0\phi(s_\varepsilon,\lambda_\varepsilon))\\
        \leq& H^N(\mu_\varepsilon,\frac{2}{\varepsilon}\sum_{k=0}^\infty \frac{1}{2^k q_k}\langle \phi_k,\mu_\varepsilon-\lambda_\varepsilon\rangle\phi_k',\frac{2}{\varepsilon}\sum_{k=0}^\infty \frac{1}{2^k q_k}\langle \phi_k,\mu_\varepsilon-\lambda_\varepsilon\rangle\phi_k'')\\
        &\qquad-H^N(\lambda_\varepsilon,\frac{2}{\varepsilon}\sum_{k=0}^\infty \frac{1}{2^k q_k}\langle \phi_k,\mu_\varepsilon-\lambda_\varepsilon\rangle\phi_k',\frac{2}{\varepsilon}\sum_{k=0}^\infty \frac{1}{2^k q_k}\langle \phi_k,\mu_\varepsilon-\lambda_\varepsilon\rangle\phi_k'')\\
        \leq& C_5\phi_\varepsilon(t_\varepsilon,s_\varepsilon,\mu_\varepsilon,\lambda_\varepsilon),
    \end{aligned}
\end{equation}
where $C_5$ is a constant independent of $\varepsilon$. Let $\varepsilon\to 0$ we have
\begin{equation}
    M=U^N(\bar{t},\bar{\mu})-W^N(\bar{t},\bar{\mu})-\varepsilon_0\phi(\bar{t},\bar{\mu})\leq 0. 
\end{equation}
This contradicts Equation \ref{equation33}.
\section{Conclusion}

In this paper, we investigated an Itô-type formula for measure-valued processes, which can be interpreted either as solutions to certain stochastic partial differential equations (SPDEs) or as decompositions akin to the classical semimartingale decomposition. The primary contributions of this work are as follows:

In the first part, we derived an Itô-type formula for controlled superprocesses in the one-dimensional setting. This was achieved by approximating second-order linear derivatives through a linear combination of products of unary functions. Subsequently, we extended the formula to the multi-dimensional and time-dependent cases, demonstrating that it holds for other measure-valued processes with similar SPDE structures. These findings offer new insights into the analysis of SPDEs and contribute to advancing the understanding of control problems for measure-valued processes.

The second part of the paper focused on applications of the derived formula to the control problem of superprocesses. By following the dynamic programming approach presented in \cite{bib2}, we established the dynamic programming principle (DPP) and derived both the Hamilton-Jacobi-Bellman (HJB) equation and the corresponding verification theorem for the control problem. The Itô formula played a crucial role in these derivations.

In the third part, we proposed a heuristic definition for a new type of linear derivative of finite measures, and we introduced the notion of a viscosity solution to an equation involving these derivatives. We demonstrated that the value function is the viscosity solution to the HJB equation in this context and proved the uniqueness of the viscosity solution when the second derivative of finite measures is absent.

\bibliographystyle{plain}

\end{document}